\documentclass[reqno,centertags,11pt,a4paper]{amsart}
\usepackage{amssymb,mathrsfs}
\usepackage{color,umoline}
\usepackage[dvipsnames]{xcolor}
\usepackage{graphicx}
\usepackage{enumitem}
\usepackage[font=small,labelfont=bf]{caption}
\usepackage{tikz,subcaption}
\usepackage{relsize}
\usepackage[mathscr]{eucal}
\usepackage{verbatim}
\usepackage[draft]{todonotes}
\usepackage[
  colorlinks=true,
  linkcolor=blue,
  citecolor=blue
]{hyperref}
\usetikzlibrary{patterns}
\tikzstyle{EDR}=[draw=lightgray,line width=0pt,preaction={clip, postaction={pattern=north east lines, pattern color=gray}}]
\tikzstyle{EDR1}=[draw=lightgray,line width=0pt,preaction={clip, postaction={pattern=north west lines, pattern color=gray}}]

\definecolor{mygray}{gray}{0.95}

\definecolor{mypink1}{rgb}{1.2,1.1,0.9}

\definecolor{mypink2}{rgb}{1.0,0.95 ,0.9}

\definecolor{mypink3}{rgb}{1.0,0.6,0.7}

\numberwithin{equation}{section}

\newtheorem{theorem}{Theorem}[section]

\newtheorem{lemma}[theorem]{Lemma}

\newtheorem{proposition}[theorem]{Proposition}
\newtheorem{remark}[theorem]{Remark}

\newcommand{\A}{\mathcal A}
\newcommand{\Pcal}{\mathcal P}
\newcommand{\M}{\mathcal M}

\newcommand{\Aalg}{\mathcal A_{\mathbb Z}(\mathcal M)}

\newcommand{\Torus}{\mathbb T}
\newcommand{\Z}{\mathbb Z}
\newcommand{\N}{\mathcal N}
\newcommand{\Q}{\mathbb Q}
\newcommand{\R}{\mathbb R}
\newcommand{\F}{\mathcal F}

\newcommand{\supp}{\operatorname{supp}}
\newcommand{\tr}{\tau}
\newcommand{\one}{\mathbf 1}
\newcommand{\E}{\mathbb E}
\newcommand{\id}{\operatorname{id}}

\newcommand{\beq}{\begin{equation}}
	\newcommand{\eeq}{\end{equation}}
\newcommand{\beqq}{\begin{equation*}}
	\newcommand{\eeqq}{\end{equation*}}
\newcommand{\ben}{\begin{eqnarray}}
	\newcommand{\een}{\end{eqnarray}}
\newcommand{\beno}{\begin{eqnarray*}}
	\newcommand{\eeno}{\end{eqnarray*}}

\begin{document}
\setcounter{page}{1}
\title[Noncommutative maximal inequalities for polynomial ergodic averages]
{Noncommutative maximal inequalities for polynomial ergodic averages}

\author[]{Guixiang Hong}
\address{
Institute for Advanced Study in Mathematics\\
Harbin Institute of Technology\\
Harbin
150001\\
China}
\email{gxhong@hit.edu.cn}

\author[]{Wenbo Li}
\address{
Institute for Advanced Study in Mathematics\\
Harbin Institute of Technology\\
Harbin
150001\\
China}
\email{lwbmath@hit.edu.cn}

\author[]{\'{E}ric Ricard}
\address{UNICAEN\\ CNRS\\ LMNO\\ 14000 Caen\\France}
\email{eric.ricard@unicaen.fr}

\author[]{Liang Wang}
\address{Department of Mathematics\\ City University of Hong Kong\\Hong Kong SAR}
\email{L.Wang@cityu.edu.hk}

\thanks{}


\date{\today}

\begin{abstract}
We prove a noncommutative maximal ergodic inequality for averages
along polynomial sequences. Let $\gamma$ be a
trace-preserving automorphism of a semifinite von Neumann algebra $(\mathcal N,\tau)$.
We show that the associated polynomial averages
\begin{equation*}
 A_Nf:=\frac1N\sum_{n=1}^N\gamma^{P(n)}(f),
 \qquad N\in\mathbb N,
\end{equation*}
satisfy a strong
maximal inequality on $L_p(\mathcal N)$ for every $1<p<\infty$, extending the
previously known restricted range of $p$.

The proof follows Bourgain's major-arc strategy but requires substantially new ideas and tools for operators that may have further applications in noncommutative analysis. More precisely, we obtain a localized maximal inequality by developing a noncommutative version of Stein's extrapolation, novel combinatorial methods, and a surprising multilinear version of Doob's maximal inequality. For the required decaying $L_2$-approximation, we combine two of Bourgain's constructions in a way that avoids the multi-frequency maximal inequality used in the scalar proof.
\end{abstract}

\maketitle
\setcounter{tocdepth}{2}
\tableofcontents

\section{Introduction and main result}

Motivated in part by quantum physics, noncommutative analysis has
developed rapidly. Noncommutative ergodic theory, an important part
of this subject, dates back to the early study of ``rings of
operators''. Initially, the focus was mainly on mean ergodic
theorems. Lance \cite{La} initiated the study of individual ergodic
theorems by establishing a noncommutative analogue of Birkhoff's
ergodic theorem for bounded functions. This was followed by a
series of further results; see, for example, \cite{CoN,Ku}.
Thirty years after Yeadon's weak type $(1,1)$ maximal ergodic
inequality \cite{Ye}, Junge and Xu \cite{JX07} established
noncommutative Birkhoff's ergodic theorems on $L_p$-spaces. Their approach
drew on the Pisier/Junge theory of vector-valued noncommutative $L_p$-spaces \cite{Pisier+Asterique98}, operator space theory, and
noncommutative Doob's maximal inequalities \cite{Junge+Crelle2002}.
Their work has inspired many subsequent studies of noncommutative
maximal inequalities and ergodic theorems; see, for example,
\cite{CAD,Be,HRW,Hu,LS,Ho,HoSu18,LeXu12,HLW,CW,Cadilhac,
CadilhacRicard+PM2025,HongRicard+PTRF2026}.
Further background can be found in the introductions of \cite{HoSu18,CHW24} and the
references therein.
 
Despite these advances, progress on noncommutative analogues of
Bourgain's generalizations of Birkhoff's ergodic theorems from the whole natural numbers to arithmetic subsets
\cite{Bo881,Bo882,Bo89} has been made only recently. Chen, Hong,
and Wang \cite{CHW24} initiated this study and obtained a maximal
ergodic inequality for averages along powers in the range
\begin{equation}\label{restricted p}
 p>\frac{1+\sqrt5}{2}.
\end{equation}
For averages along the primes, Hong and Wang
\cite{HongWang+arXiv2026} subsequently removed this restriction
and established bilateral almost uniform convergence. A related
result for random ergodic averages can be found in
\cite{LeMerdyZadeh+arXiv2026}.
In the present paper, we remove the restriction
\eqref{restricted p} for all integer  coefficients polynomials and thus establish
a full noncommutative analogue of Bourgain's seminal theory.

Let $\mathcal N$ be a von Neumann algebra equipped with a normal
faithful semifinite trace $\tau$, and let $\gamma$ be a
trace-preserving automorphism of $\mathcal N$, that is, $\tau\circ\gamma=\tau$.

We denote by $L_p(\mathcal N)$ the noncommutative $L_p$-space
associated with $(\mathcal N,\tau)$. It is well known that
$\gamma$ extends to an isometry on $L_p(\mathcal N)$ for every
$1\leq p\leq\infty$. A detailed definition of noncommutative $L_p$-space as well as the vector-valued spaces $L_p(\mathcal N; \ell_\infty)$ and $L_p(\mathcal N; \ell_1)$ will be recalled in next section.

Let
\[
 P(n)=\sum_{i=1}^d b_i n^i
\]
be a polynomial with integer coefficients and $b_d>0$. We consider the associated ergodic averages
\begin{equation*}
 A_Nf:=\frac1N\sum_{n=1}^N\gamma^{P(n)}(f),
 \qquad N\in\mathbb N.
\end{equation*}

\begin{theorem}\label{thm1}
Let $1<p<\infty$. Then
\begin{equation}\label{mthm1}
 \big\|(A_Nf)_{N\in\mathbb N}\big\|_
 {L_p(\mathcal N;\ell_\infty)}
 \lesssim_{p,P} \|f\|_{L_p(\mathcal N)}
\end{equation}
for every $f\in L_p(\mathcal N)$. Moreover, $A_Nf$ converges
bilaterally almost uniformly to an element of $L_p(\mathcal N)$
as $N\to\infty$.
\end{theorem}

The main task of this paper is to prove the maximal inequality
\eqref{mthm1}. For bilateral almost uniform convergence, we adapt
the argument in \cite[Section~4 and Remark~4.9]{HongWang+arXiv2026}
from $P(n)=n^d$ to general polynomials with integer coefficients.
Together with \eqref{mthm1}, this gives convergence throughout
the range $1<p<\infty$. We outline this argument at the end of
Section~\ref{sec:proof of main result}.

As in \cite[Section~4.2]{CHW24}, a standard transference argument
reduces the proof of Theorem~\ref{thm1} to the corresponding
maximal inequality for the shift action on $\mathbb Z$. Write
\[
 \mathcal A_{\mathbb Z}(\mathcal N)
 :=\ell_\infty(\mathbb Z)\overline\otimes\mathcal N,
\]
equipped with the canonical tensor-product trace. For every
finitely supported $\mathcal S(\mathcal N)$-valued function $f$ on $\mathbb Z$ where $\mathcal S(\mathcal N)$ is the set of bounded operators with finite support trace,
define
\begin{equation*}
 \mathcal A_Nf:=f*K_N,
 \qquad
 K_N:=\frac1N\sum_{n=1}^N\delta_{P(n)}.
\end{equation*}
Here convolution is taken in the $\mathbb Z$-variable.

By a standard doubling argument, it suffices to prove the
following dyadic maximal inequality:
\begin{equation}\label{thm11}
 \big\|(\mathcal A_{2^k}f)_{k\in\mathbb N}\big\|_
 {L_p(\mathcal A_{\mathbb Z}(\mathcal N);\ell_\infty)}
 \lesssim_{p,P}
 \|f\|_{L_p(\mathcal A_{\mathbb Z}(\mathcal N))}
\end{equation}
for every finitely supported $\mathcal S(\mathcal N)$-valued
function $f$ on $\mathbb Z$.

A key ingredient in the proof of Theorem~\ref{thm1}, beyond
the results of \cite{CHW24}, is the following noncommutative
version of Bourgain's estimate \cite[Lemma~7.32]{Bo89}.

\begin{proposition}\label{Key1-es}
Let $1<p<\infty$ and $N_0\geq3$. Then
\begin{equation*}
 \big\|(f*K_N)_{N_0<N<N_0^2}\big\|_
 {L_p(\mathcal A_{\mathbb Z}(\mathcal N);\ell_\infty)}
 \lesssim_{p,P} \log\log N_0\,
 \|f\|_{L_p(\mathcal A_{\mathbb Z}(\mathcal N))}
\end{equation*}
for every finitely supported $\mathcal S(\mathcal N)$-valued
function $f$ on $\mathbb Z$.
\end{proposition}

Noncommutativity creates lots of difficulties in adapting scalar
pointwise arguments, since many fundamental classical notions and techniques, such as maximal functions, stopping times, the Borel--Cantelli/covering lemma, and monotonicity or
convexity arguments, either fail to exist or cease to hold in the noncommutative framework. This makes
research in this area highly challenging, having led researchers to introduce new
concepts, tools, and techniques, including vector-valued noncommutative $L_p$-spaces and probabilistic methods as in the previously-mentioned papers \cite{La,Ye,Pisier+Asterique98,Junge+Crelle2002,JX07}

The proof of Proposition \ref{Key1-es}  requires entirely new ideas and tools to overcome certain difficulties that are specific to the noncommutative case but, in the commutative case, appear as utterly trivial facts.

i) First, we develop a noncommutative version of Stein's
extrapolation argument in Section~\ref{sec:ste}. Building on
earlier ideas of Kolmogorov and Calder\'on, Stein
\cite{Stein+AM1961} showed, under suitable hypotheses, that
almost everywhere convergence of a sequence of convolution
operators on $L_p$ implies a weak type $(p,p)$ estimate for the
associated maximal operator when $1\leq p\leq2$. In this sense,
his result is an analogue of the Uniform Boundedness Principle
for almost everywhere convergence rather than norm convergence.
Similar results were obtained by Burkholder \cite{Bur64} in
probability and by Sawyer \cite{Saw} in ergodic theory.

Noncommutative individual ergodic theorems, including
\cite{La,Ye,JX07,HLW}, have generally been proved with the help
of maximal inequalities. This makes it natural to ask whether
Stein's implication has a noncommutative counterpart. A general
analogue is an extremely difficult problem: Stein's proof relies
on a pointwise maximal function and many point-dependent techniques including a Borel--Cantelli argument,
neither of which can be used directly in the noncommutative
setting. In the Appendix, we obtain an $L_1$ result in the
spirit of Burkholder's theorem. For $1<p\leq2$,
Theorem~\ref{pro:stein extrapolation} gives a weaker form of
Stein's result that is sufficient for the present paper. A key
step is to find a suitable replacement for the
Borel--Cantelli argument.
In the commutative proof, one constructs sets $E_k$ on
which the corresponding maximal functions exceed
$\lambda_k$, where $\lambda_k\to\infty$.
The Borel--Cantelli argument ensures that almost every
point belongs to infinitely many $E_k$, which is the
key to the contradiction.

In our proof, these sets are replaced by projections, and the Borel--Cantelli type result in Proposition~\ref{prop:borel-cantelli} plays the corresponding
role. We first use a tensor-product construction to make the
relevant projections commute. Applying the conditional expectation onto the
commutative algebra generated by these projections then yields the scalar inequalities corresponding to
the operator inequalities in \eqref{est7}. Consequently, the commutative  contradiction
argument completes the proof of
Theorem~\ref{pro:stein extrapolation}.  

ii) Second, the moment expansion in
Section~\ref{combinatorial method} presents the second substantial
obstacle in the noncommutative setting. After applying
extrapolation, Bourgain uses duality to reduce part of his
argument to an estimate of
\[
 \left\|\sum_{i\in I}a_i\right\|_q^q
\]
for positive scalars $a_i$ and a large integer $q$. He
expands the $q$th power and uses the facts that the factors
commute and that their products remain nonnegative; see
\cite[(7.40)--(7.41)]{Bo89}. Neither fact holds for positive
operators, so this part of the scalar proof cannot be carried
over directly. To overcome this difficulty, we combine
M\"obius inversion with probabilistic and interpolation
arguments to control the terms with repeated indices. The
result is Proposition~\ref{prop:rep}, which plays the role of
the scalar estimate in \cite[(7.41)]{Bo89}. 

iii) Third, the terms with distinct indices present the third
substantial obstacle. In the scalar proof, Bourgain uses
\cite[(7.42)]{Bo89} and the positivity of products to obtain
\cite[(7.43)]{Bo89}. In the noncommutative setting, a product
of positive operators need not be positive, and we no longer
have an exact analogue of \cite[(7.40)]{Bo89}.

Instead, cyclicity of the trace gives the identity
\eqref{eq:cyclic-minimum}. Duality between
$L_q(\N;\ell_1)$ and
$L_{\frac{q}{q-1}}(\N;\ell_\infty)$ then reduces the problem
to estimating $ \big\|(D_i)_i\big\|_
 {L_{\frac{q}{q-1}}(\N;\ell_\infty)}
$.
A direct product estimate can fail in the noncommutative
setting; see Remark~\ref{extreme case}. Using the combinatorial
ideas developed above, we express the relevant terms in $D_i$
through linear combinations of products of martingale averages and establish the
multilinear noncommutative Doob's maximal inequality
\eqref{prop:Multi-Doob-mart} to control them. Such a maximal
inequality for products of martingales is rather unexpected in
the noncommutative setting and may be of independent interest.
The details are given in
Sections~\ref{subsec:maximal} and~\ref{subsec:multilinear}.

\bigskip

We next explain how Proposition~\ref{Key1-es} enters the proof
of Theorem~\ref{thm1}. The proof follows the overall strategy
of Bourgain \cite{Bo89}; see
Section~\ref{sec:proof of main result}. It relies on maximal
and approximation estimates for the major-arc multipliers.
Combining Proposition~\ref{Key1-es} with the results of
\cite{CHW24} gives the $L_p$ estimates \eqref{Key-es5} and
\eqref{Key-es6}, whose bounds have at most polynomial growth. To make
the resulting estimates summable, via interpolation we also need the
exponentially decaying $L_2$ estimates in
Lemma~\ref{lem:L2decay-v2}.
The main additional difficulty is to prove these $L_2$
estimates in the noncommutative setting. In the scalar case,
Bourgain obtains the corresponding estimates using the
multi-frequency maximal inequality \cite[Lemma~4.1]{Bo89}.
Its pointwise metric-entropy argument does not directly extend
to $L_2(\Aalg;\ell_\infty)$; see
Remark~\ref{rem:Bourgain-Lemma41}. We use a different argument
based on the decomposition in \cite[Sections~3 and~6]{Bo881}.
The required estimates are established in
Lemmas~\ref{lem:layer-v2} and~\ref{lem:comparison-v2}.

More precisely, in \cite[Section~6]{Bo89}, Bourgain divides
the set of rational numbers $a/q$, with $(a,q)=1$, into subsets according to the size of $q$, placing those with $2^s\leq q<2^{s+1}$ in the same set. A central advance of that paper is the multi-frequency maximal inequality \cite[Lemma~4.1]{Bo89}, which controls the corresponding
maximal operators without using a common denominator for
the rational points. Its bound depends only logarithmically
on the number of points, see
Remark~\ref{rem:Bourgain-Lemma41}.
This approach makes the proof less dependent on the
arithmetic properties of the exponential sums, unlike
the earlier argument in \cite[Sections~3 and~6]{Bo881}.

Establishing a noncommutative analogue of this
multi-frequency maximal inequality appears to be a difficult
problem. We instead combine Bourgain's earlier decomposition
\cite[Sections~3 and~6]{Bo881}, presented here in
\eqref{eq:Qu-v2}--\eqref{eq:LN-v2}, with the noncommutative
sampling principle in Lemma~\ref{lem:CHW-v2}.
This decomposition divides the rational numbers into sets
according to common denominators, rather than only the
size of their denominators. We apply the sampling principle
to estimate the contribution of each such set and then sum
the estimates using the decay of the exponential sums
$S(a/q)$, defined in \eqref{eq:S-v2}.
This gives the operator-valued estimate in
Lemma~\ref{lem:layer-v2}.
The argument retains the stronger arithmetic information
used in Bourgain's earlier proof, including Weil's estimate
and the multiplicative properties of the exponential sums.
These enter through the scalar approximation
\eqref{eq:scalar-error-v2}, whose error is
$O_P((\log N)^{-\sigma})$ for some $\sigma>1/2$.

A key new step is Lemma~\ref{lem:comparison-v2}, which
compares the truncated multiplier $L_{2^k}^{(\leq2^s)}$
with $\widehat M_{s,k}$ in the noncommutative maximal
$L_2$-norm. The latter multiplier is defined in
\eqref{M_k}, following Bourgain's construction in
\cite[Section~7]{Bo89}.
This comparison allows us to obtain the decaying
$L_2$-estimates in Lemma~\ref{lem:L2decay-v2} from
Lemma~\ref{lem:layer-v2} and the scalar approximation
\eqref{eq:scalar-error-v2}, without requiring a
noncommutative analogue of \cite[Lemma~4.1]{Bo89}.

\noindent\textbf{Notation.}
Throughout the paper, $C$ denotes a positive constant whose value
may change from line to line. Constants may depend on the fixed
polynomial $P$ and the fixed cutoff functions, but are independent
of the averaging and scale parameters unless otherwise stated.
For nonnegative quantities $X$ and $Y$, we write $X\lesssim Y$
if $X\leq CY$. Additional dependence is indicated by subscripts,
as in $X\lesssim_{\varepsilon}Y$. Finally, $X\simeq Y$ means
that both $X\lesssim Y$ and $Y\lesssim X$ hold.

\section{Preliminaries}\label{pre}

\subsection{Noncommutative \texorpdfstring{$L_p$}{Lp}-spaces}

We refer to \cite{DDPS}. Let $(\mathcal M,\tau)$ be a semifinite von Neumann algebra equipped
with a normal faithful semifinite trace $\tau$. Denote by
$\mathcal S(\mathcal M)_+$ the set of all $x\in\mathcal M_+$ such
that
\[
 \tau(\supp x)<\infty,
\]
where $\supp x$ denotes the support projection of $x$. The linear
span of $\mathcal S(\mathcal M)_+$, denoted by
$\mathcal S(\mathcal M)$, is a weak$^*$-dense $*$-subalgebra of
$\mathcal M$.
For $1\leq p<\infty$, define
\[
 \|x\|_p:=\bigl(\tau(|x|^p)\bigr)^{1/p},
 \qquad x\in\mathcal S(\mathcal M),
\]
where $|x|=(x^*x)^{1/2}$. The completion of
$\mathcal S(\mathcal M)$ with respect to this norm is the
noncommutative $L_p$-space associated with $(\mathcal M,\tau)$,
denoted by $L_p(\mathcal M)$. As usual, we set
\[
 L_\infty(\mathcal M):=\mathcal M,
 \qquad
 \|x\|_\infty:=\|x\|_{\mathcal M}.
\]

Suppose that $\mathcal M$ acts on a separable Hilbert space $H$.
A closed densely defined operator $x$ on $H$ is said to be affiliated with
$\M$ if $u^* xu = x$ for all unitary operators $u$ in the commutant $\M'$	 of $\M$. If $x$ is a densely defined self-adjoint operator on $H$ with spectral decomposition $x=\int_{-\infty}^\infty sde_s^x$, then,
for any Borel subset $B\subset \R$, we denote by $\chi_B(x)$ the corresponding spectral projection $\int_{-\infty}^\infty \chi_B sde_s^x$.
A closed and densely defined operator $x$ affiliated with
$\M$ is called $\tau$-measurable if there exists $\lambda>0$ such that
$$\tau(\chi_{(\lambda,\infty)}(|x|))<\infty.$$
We denote the set of the $*$-algebra of $\tau$-measurable operators by $L_0(\M)$. 

Let $(X,\mu)$ be a measure space and let $\mathcal B$ be a
semifinite von Neumann algebra equipped with a normal faithful
semifinite trace $\tau_{\mathcal B}$. We use the notation
\[
 \Big(\mathcal A_X(\mathcal B), \tau\Big)
 :=\Big(L_\infty(X)\overline\otimes\mathcal B,\int_X\otimes\tau_{\mathcal B}\Big)
\]
for the associated semicommutative von Neumann algebra equipped with the tensor-product trace.

\subsection{Noncommutative
\texorpdfstring{$\ell_\infty$-valued $L_p$}{ell-infinity-valued Lp}-spaces}

Let $1\leq p\leq\infty$ and let $\Lambda$ be an index set. The
space $L_p(\mathcal M;\ell_\infty(\Lambda))$ was fristly introduced  by Pisier \cite{Pisier+Asterique98} and then generalized by Junge \cite{Junge+Crelle2002} as a right functional setting to avoid talking about maximal funcions. It consists of all
families
\[
 x=(x_\lambda)_{\lambda\in\Lambda},
 \qquad x_\lambda\in L_p(\mathcal M),
\]
for which there exist $a,b\in L_{2p}(\mathcal M)$ and a bounded
family $(y_\lambda)_{\lambda\in\Lambda}\subset\mathcal M$ such that
\[
 x_\lambda=ay_\lambda b,
 \qquad \lambda\in\Lambda.
\]
Here, when $p=\infty$, we use the convention
$L_{2p}(\mathcal M)=L_\infty(\mathcal M)=\mathcal M$.

The norm on $L_p(\mathcal M;\ell_\infty(\Lambda))$ is defined by
\begin{equation*}
 \big\|(x_\lambda)_{\lambda\in\Lambda}\big\|_
 {L_p(\mathcal M;\ell_\infty(\Lambda))}
 :=
 \inf_{x_\lambda=ay_\lambda b}
 \left\{
 \|a\|_{2p}
 \sup_{\lambda\in\Lambda}\|y_\lambda\|_\infty
 \|b\|_{2p}
 \right\},
\end{equation*}
where the infimum is taken over all such factorizations.

It was proved in \cite{JX07} that a family
$(x_\lambda)_{\lambda\in\Lambda}$ belongs to
$L_p(\mathcal M;\ell_\infty(\Lambda))$ if and only if
\[
 \sup_{\substack{I\subset\Lambda\\ I\text{ finite}}}
 \big\|(x_\lambda)_{\lambda\in I}\big\|_
 {L_p(\mathcal M;\ell_\infty(I))}
 <\infty.
\]
Moreover,
\begin{equation*}
 \big\|(x_\lambda)_{\lambda\in\Lambda}\big\|_
 {L_p(\mathcal M;\ell_\infty(\Lambda))}
 =
 \sup_{\substack{I\subset\Lambda\\ I\text{ finite}}}
 \big\|(x_\lambda)_{\lambda\in I}\big\|_
 {L_p(\mathcal M;\ell_\infty(I))}.
\end{equation*}

If every $x_\lambda$ is self-adjoint, then
$(x_\lambda)_{\lambda\in\Lambda}$ belongs to
$L_p(\mathcal M;\ell_\infty(\Lambda))$ if and only if there exists
$a\in L_p(\mathcal M)_+$ such that
\[
 -a\leq x_\lambda\leq a,
 \qquad \lambda\in\Lambda.
\]
In this case,
\begin{equation*}
 \big\|(x_\lambda)_{\lambda\in\Lambda}\big\|_
 {L_p(\mathcal M;\ell_\infty(\Lambda))}
 =
 \inf\left\{
 \|a\|_p:
 a\in L_p(\mathcal M)_+,\
 -a\leq x_\lambda\leq a
 \text{ for every }\lambda\in\Lambda
 \right\}.
\end{equation*}

In particular, if $x_\lambda\geq0$ for every $\lambda\in\Lambda$,
then
\begin{equation*}
 \big\|(x_\lambda)_{\lambda\in\Lambda}\big\|_
 {L_p(\mathcal M;\ell_\infty(\Lambda))}
 =
 \inf\left\{
 \|a\|_p:
 a\in L_p(\mathcal M)_+,\
 x_\lambda\leq a
 \text{ for every }\lambda\in\Lambda
 \right\}.
\end{equation*}
This justifies the commonly used notation
$$\big\|(x_\lambda)_{\lambda\in\Lambda}\big\|_
 {L_p(\mathcal M;\ell_\infty(\Lambda))}
 =\big\| \mathop{\rm sup^+}\limits_{\lambda\in \Lambda} x_\lambda\big\|_p.$$

When the index set is clear from the context, we simply write
$L_p(\mathcal M;\ell_\infty)$ instead of
$L_p(\mathcal M;\ell_\infty(\Lambda))$.

We shall also use freely  the following interpolation result from
\cite{JX07}.

\begin{lemma}
Let $1\leq p_0,p_1\leq\infty$ and $0<\theta<1$. Define $p$ by
\[
 \frac{1}{p}
 =
 \frac{1-\theta}{p_0}
 +
 \frac{\theta}{p_1}.
\]
Then
\begin{equation*}
 \big(
 L_{p_0}(\mathcal M;\ell_\infty),
 L_{p_1}(\mathcal M;\ell_\infty)
 \big)_\theta
 =
 L_p(\mathcal M;\ell_\infty)
\end{equation*}
isometrically.
\end{lemma}

It is well-known that $L_{p}(\mathcal M;\ell_\infty)$ is a dual space when $1<p\leq\infty$, and its predual space is denoted by $L_{p'}(\mathcal M;\ell_1)$ with $p'$ being the conjugate exponent of $p$. In the present paper, we do not need all the properties of $L_{p'}(\mathcal M;\ell_1)$, and only mention the only one that we will use: given a sequence of positive operators $(x_\lambda)_{\lambda\in \Lambda}$, one has
$$\|(x_\lambda)_{\lambda\in \Lambda}\|_{L_{p'}(\mathcal M;\ell_1)}=\Big\|\sum_{\lambda\in\Lambda}x_{\lambda}\Big\|_{p'}.$$

Besides the abvove noncommutative strong maximal norms, there also exist noncommutative weak maximal quasi-norms; this will be recalled in \eqref{phi_inf}.

\section{A noncommutative version of Stein's extrapolation}
\label{sec:ste}

Let $G$ be a compact abelian group equipped with the normalized Haar measure $\mu$. Let $(\mathcal M,\tau)$ be
a finite von Neumann algebra equipped with a normal faithful tracial
state. We write
\[
 \mathcal A=L_\infty(G)\overline\otimes\mathcal M,
 \qquad
 \varphi=\int_G\otimes\tau.
\]
For $g\in G$, let $\alpha_g$ be the trace-preserving automorphism of
$\mathcal A$ induced by translation:
\[
 (\alpha_gf)(h)=f(g^{-1}h),
 \qquad h\in G.
\]

Let $\mathcal I$ be an index set. For $r\in\mathcal I$, we consider the convolution operator on the group variable with a Borel probability measure $\nu_r$
\[
 A_r(f)(g)
 = \int_{G}f(gh^{-1})\,d\nu_r(h),
 \qquad g\in G.
\]
Each $A_r$ is a positive trace-preserving contraction on
$L_p(\mathcal A)$ for every $1\leq p\leq\infty$, and
\[
 A_r\alpha_g=\alpha_gA_r.
\]

We shall also use countable tensor products of finite von Neumann
algebras. If $(\mathcal M^\infty,\tau^\infty)$ is such a tensor
product, we write
\[
 \mathcal A^\infty
 =
 L_\infty(G)\overline\otimes\mathcal M^\infty,
 \qquad
 \varphi^\infty
 =
 \int_G\otimes\tau^\infty.
\]

\begin{theorem}\label{pro:stein extrapolation}
Let $1<p\leq2$. Suppose that there exists $B>0$ such that, for every
finite von Neumann algebra $(\mathcal M,\tau)$,
\begin{equation}\label{L18 bound}
 \big\|(A_r(f))_{r\in\mathcal I}\big\|_
 {L_1(\mathcal A;\ell_\infty)}
 \leq B\|f\|_{L_p(\mathcal A)},
 \qquad f\in L_p(\mathcal A).
\end{equation}
Then, for every finite von Neumann algebra $(\mathcal M,\tau)$,
\begin{equation}\label{Lambdap8 bound}
 \big\|(A_r(f))_{r\in\mathcal I}\big\|_
 {\Lambda_p(\mathcal A;\ell_\infty)}
 \leq CB\|f\|_{L_p(\mathcal A)},
 \qquad f\in L_p(\mathcal A),
\end{equation}
where $C$ is an absolute constant.
\end{theorem}
Since $(A_r(f))_{r\in \mathcal I}$ has a bounded cardinal, we can reduce to von Neumann algebras $\mathcal M$ with a bounded cardinality, so that the statement is really about sets.

The symbol $\mathrm{sup}^{+}$ appearing in the definitions of
$\Phi_{\inf}$ and $\Phi_{\sup}$ below is part of the notation for
the corresponding distribution functions.

For a family $(f_n)_n\subset L_0(\mathcal A)_+$, define
\[
 \big\|(f_n)_n\big\|_{\Lambda_p(\mathcal A;\ell_\infty)}
 :=
 \sup_{\lambda>0}
 \lambda\,
 \Phi_{\inf}\left(
 \mathop{\mathrm{sup}^{+}}\limits_n f_n>\lambda
 \right)^{1/p},
\]
where
\begin{equation}\label{phi_inf}
 \Phi_{\inf}\left(
 \mathop{\mathrm{sup}^{+}}\limits_n f_n>\lambda
 \right)
 :=
 \inf\left\{
 \varphi(e^\perp):
 \begin{array}{l}
 e\in\mathcal P(\mathcal A),\\
 \|ef_ne\|_\infty\leq\lambda
 \text{ for every }n
 \end{array}
 \right\}.
\end{equation}
Here, $\mathcal P(\mathcal A)$ denotes the set of projections in
$\mathcal A$.

For positive families, we shall also use the quantity introduced
by Cadilhac \cite{Cadilhac}:
\begin{equation*}
\begin{aligned}
 \Phi_{\sup}\left(
 \mathop{\mathrm{sup}^{+}}\limits_n f_n>\lambda
 \right)
 :=
 \sup\Biggl\{
 &\sum_k\varphi(e_k):
 (e_k)_k\subset\mathcal P(\mathcal A)
 \text{ is pairwise orthogonal},\\
 &\text{and for every }k\text{ there exists }n
 \text{ such that }
 e_kf_ne_k>\lambda e_k
 \Biggr\}.
\end{aligned}
\end{equation*}

\begin{lemma}\label{sup>inf}
Let $(f_n)_n\subset L_0(\mathcal A)_+$. Then, for every $\lambda>0$,
\[
 \Phi_{\sup}\left(
 \mathop{\mathrm{sup}^{+}}\limits_n f_n>\lambda
 \right)
 \geq
 \Phi_{\inf}\left(
 \mathop{\mathrm{sup}^{+}}\limits_n f_n>\lambda
 \right).
\]
\end{lemma}

If $\mathcal A$ is commutative, the inequality above becomes an
equality.

\begin{proof}
By Zorn's lemma, choose a maximal pairwise orthogonal family
$(e_k)_k$ such that, for every $k$, there exists an index $n(k)$
satisfying
\[
 e_kf_{n(k)}e_k>\lambda e_k.
\]
Since $\varphi$ is finite and faithful, the nonzero members of this
family form an at most countable collection. Put
\[
 e=\left(\bigvee_ke_k\right)^\perp.
\]
We claim that
\[
 ef_ne\leq\lambda e
 \qquad\text{for every }n.
\]
Indeed, suppose that this fails for some $n_0$. Then the spectral
projection
\[
 e'
 =
 \mathbf 1_{(\lambda,\infty)}(ef_{n_0}e)
\]
is a nonzero subprojection of $e$ and satisfies
\[
 e'f_{n_0}e'>\lambda e'.
\]
Since $e'$ is orthogonal to every $e_k$, adjoining $e'$ to the
family $(e_k)_k$ contradicts maximality.

Consequently, $e$ is admissible in the definition of
$\Phi_{\inf}$, and hence
\[
 \Phi_{\inf}\left(
 \mathop{\mathrm{sup}^{+}}\limits_n f_n>\lambda
 \right)
 \leq\varphi(e^\perp)
 =\sum_k\varphi(e_k)
 \leq
 \Phi_{\sup}\left(
 \mathop{\mathrm{sup}^{+}}\limits_n f_n>\lambda
 \right).
\]
\end{proof}

We next establish an ergodic version of the Borel--Cantelli lemma largely inspired by Sawyer \cite{Saw}.

\begin{proposition}\label{prop:borel-cantelli}
Let $\mathcal M$ be a $\mathrm{II}_1$ factor and let
$\mathcal U$ be a subgroup of the unitary group of $\mathcal M$
such that
\[
 \mathcal U' \cap\mathcal M=\mathbb C1.
\]
Let $(e_k)_{k\geq1}$ be a sequence of projections in
$ \mathcal A=L_\infty(G)\overline\otimes\mathcal M$ such that
\[
 \sum_{k\geq1}\varphi(e_k)=\infty.
\]
Then there exist $(g_k)_{k\geq1}\subset G$ and
$(u_k)_{k\geq1}\subset\mathcal U$ such that, for every $\ell\geq1$,
\[
 \bigvee_{k\geq\ell}
 \alpha_{g_k}(u_k^*e_ku_k)
 =1.
\]
\end{proposition}

We postpone the proof of Proposition
\ref{prop:borel-cantelli} until the end of this section and first
prove Theorem \ref{pro:stein extrapolation}.

\begin{proof}[Proof of Theorem \ref{pro:stein extrapolation}]
By homogeneity, replacing $A_r$ by $B^{-1}A_r$, we may assume that
$B=1$. By decomposing an arbitrary element into four positive
parts, it is enough to prove the required estimate for
$f\in L_p(\mathcal A)_+$.

In view of Lemma \ref{sup>inf}, it suffices to prove that
\[
 \sup_{\lambda>0}
 \lambda^p
 \Phi_{\sup}\left(
 \mathop{\mathrm{sup}^{+}}\limits_{r\in\mathcal I}A_r(f)>\lambda
 \right)
 \lesssim1
\]
whenever $f\in L_p(\mathcal A)_+$ and $\|f\|_p\leq1$.

Define
\[
 M(\lambda)
 :=
 \sup_{\substack{f\in L_p(\mathcal A)_+\\\|f\|_p\leq1}}
 \lambda^p
 \Phi_{\sup}\left(
 \mathop{\mathrm{sup}^{+}}\limits_{r\in\mathcal I}A_r(f)>\lambda
 \right).
\]
Since $\varphi(1)=1$, we have
\[
 M(\lambda)\leq\lambda^p.
\]
Thus, if $M$ is unbounded, then necessarily
\[
 \limsup_{\lambda\to\infty}M(\lambda)=\infty.
\]

Assume, toward a contradiction, that $M$ is unbounded. The standard
selection argument used in \cite[p.~78]{Bur64} provides sequences
$(a_k)_{k\geq1}$ and $(\lambda_k)_{k\geq1}$ of positive numbers such
that
\[
 \lambda_k\longrightarrow\infty,
 \qquad
 \sum_{k\geq1}a_k^p\leq1,
\]
and
\[
 \sum_{k\geq1}
 \left(\frac{a_k}{\lambda_k}\right)^p
 M\left(\frac{\lambda_k}{a_k}\right)
 =\infty.
\]

For every $k$, choose $f_k\in L_p(\mathcal A)_+$ with
$\|f_k\|_p\leq1$ such that
\[
 \Phi_{\sup}\left(
 \mathop{\mathrm{sup}^{+}}\limits_{r\in\mathcal I}A_r(f_k)>
 \frac{\lambda_k}{a_k}
 \right)
 \geq
 \left(\frac{a_k}{\lambda_k}\right)^p
 M\left(\frac{\lambda_k}{a_k}\right)-2^{-k}.
\]
It follows that
\begin{equation}\label{est1}
 \sum_{k\geq1}
 \Phi_{\sup}\left(
 \mathop{\mathrm{sup}^{+}}\limits_{r\in\mathcal I}A_r(f_k)>
 \frac{\lambda_k}{a_k}
 \right)
 =\infty.
\end{equation}

For every $k$ we may choose a pairwise
orthogonal family $(e_{k,i})_i$ such that
\[
 \sum_i\varphi(e_{k,i})
 \geq
 \Phi_{\sup}\left(
 \mathop{\mathrm{sup}^{+}}\limits_{r\in\mathcal I}A_r(f_k)>
 \frac{\lambda_k}{a_k}
 \right)-2^{-k},
\]
and such that, for every $i$, there exists
$r_{k,i}\in\mathcal I$ satisfying
\begin{equation}\label{est2}
 e_{k,i}A_{r_{k,i}}(f_k)e_{k,i}
 >
 \frac{\lambda_k}{a_k}e_{k,i}.
\end{equation}
Set
\[
 e_k=\sum_i e_{k,i}.
\]
Equation \eqref{est1} then gives
\begin{equation}\label{est3}
 \sum_{k\geq1}\varphi(e_k)=\infty.
\end{equation}

We now enlarge the underlying von Neumann algebra. Let
\[
 \mathcal N=\mathcal M*L(\mathbb F_2)
\]
be the reduced free product with respect to the given traces. Then
$\mathcal N$ is a finite factor containing $\mathcal M$ as a
trace-preserving von Neumann subalgebra; see
\cite{TakesakiIII}. Let
\[
 \mathcal N^\infty
 =
 \overline{\bigotimes}_{k\geq1}\mathcal N,
 \qquad
 \mathcal A^\infty
 =
 L_\infty(G)\overline\otimes\mathcal N^\infty.
\]

Let
\[
 \Pi_k
 =
 \id_{L_\infty(G)}\otimes\pi_k:
 \mathcal A\longrightarrow\mathcal A^\infty
\]
be the trace-preserving embedding into the $k$-th tensor coordinate.
Thus, for elementary tensors,
\[
 \Pi_k(f\otimes m)
 =
 f\otimes
 1\otimes\cdots\otimes1\otimes m\otimes1\otimes\cdots,
\]
where $m$ occupies the $k$-th coordinate. Put
\[
 e_k'=\Pi_k(e_k),
 \qquad
 e_{k,i}'=\Pi_k(e_{k,i}).
\]
Because the noncommutative parts of these projections lie in
distinct tensor coordinates and their $G$-components belong to the
commutative algebra $L_\infty(G)$, all the projections
$(e_{k,i}')_{k,i}$ commute whenever their first indices are
different.

Since $\Pi_k$ commutes with every $A_r$, equation \eqref{est2}
implies
\begin{equation}\label{est4}
 e_{k,i}'A_{r_{k,i}}(\Pi_k(f_k))e_{k,i}'
 >
 \frac{\lambda_k}{a_k}e_{k,i}'.
\end{equation}
Moreover, \eqref{est3} gives
\[
 \sum_{k\geq1}\varphi^\infty(e_k')=\infty.
\]

Let $\mathcal U_0$ be the subgroup of the unitary group of
$\mathcal N^\infty$ consisting of finite tensor products:
\[
 \mathcal U_0
 :=
 \left\{
 \prod_{j\in F}\pi_j(v_j):
 F\subset\mathbb N\text{ finite},
 \ v_j\in\mathcal U(\mathcal N)
 \right\}.
\]
Since $\mathcal N$ is a factor,
\[
 \mathcal U_0'\cap\mathcal N^\infty=\mathbb C1.
\]
Applying Proposition \ref{prop:borel-cantelli} to $(e_k')_k$, we
obtain sequences $(g_k)_k\subset G$ and $(u_k)_k\subset\mathcal U_0$
such that
\begin{equation}\label{est5}
 \bigvee_{k\geq\ell}
 \alpha_{g_k}(u_k^*e_k'u_k)
 =1
 \qquad\text{for every }\ell\geq1.
\end{equation}

Define
\[
 e_k''
 :=
 \alpha_{g_k}(u_k^*e_k'u_k),
 \qquad
 e_{k,i}''
 :=
 \alpha_{g_k}(u_k^*e_{k,i}'u_k).
\]
Every $u_k\in\mathcal U_0$ normalizes each tensor-coordinate
algebra $\pi_j(\mathcal N)$. In particular,
\[
 u_k^*e_k'u_k
 \in
 L_\infty(G)\overline\otimes\pi_k(\mathcal N).
\]
It follows that the projections $e_k''$, and hence all
$e_{k,i}''$ with distinct first indices, commute.

Using
\[
 A_r\alpha_{g_k}=\alpha_{g_k}A_r,
 \qquad
 A_r(u_k^*hu_k)=u_k^*A_r(h)u_k,
 \qquad
 A_r\Pi_k=\Pi_kA_r,
\]
we deduce from \eqref{est4} that
\begin{equation}\label{est7}
 e_{k,i}''
 A_{r_{k,i}}\!
 \left(\alpha_{g_k}
 (u_k^*\Pi_k(f_k)u_k)\right)
 e_{k,i}''
 >
 \frac{\lambda_k}{a_k}e_{k,i}''.
\end{equation}

Set
\[
 \gamma_k
 :=
 \alpha_{g_k}(u_k^*\Pi_k(f_k)u_k).
\]
Then $\gamma_k\geq0$ and
\[
 \|\gamma_k\|_p=\|f_k\|_p\leq1.
\]

Let $(\varepsilon_k)_k$ be the canonical independent Rademacher
variables on $\Omega_\varepsilon=\{-1,1\}^{\mathbb N}$.

Since $1<p\leq2$, the noncommutative space
$L_p(\mathcal A^\infty)$ has Rademacher type $p$; see
\cite[Corollary~5.5]{PX03} and also \cite{Fack87}. Therefore,
\begin{equation*}
 \left\|
 \sum_{k\geq1}\varepsilon_ka_k\gamma_k
 \right\|_
 {L_p(\mathcal A^\infty\overline\otimes
 L_\infty(\Omega_\varepsilon))}
 \lesssim
 \left(
 \sum_{k\geq1}a_k^p\|\gamma_k\|_p^p
 \right)^{1/p}
 \leq1.
\end{equation*}

Apply \eqref{L18 bound} in the finite von Neumann algebra
\[
 \widetilde{\mathcal A}
 =
 \mathcal A^\infty
 \overline\otimes L_\infty(\Omega_\varepsilon)
\]
to the self-adjoint element
\[
 F=\sum_{k\geq1}\varepsilon_ka_k\gamma_k.
\]
By the order description of the
$L_1(\widetilde{\mathcal A};\ell_\infty)$ norm, there exists
$b\in L_1(\widetilde{\mathcal A})_+$ such that
\[
 \|b\|_1\lesssim1
\]
and
\[
 -b
 \leq
 \sum_{k\geq1}\varepsilon_ka_kA_r(\gamma_k)
 \leq b
 \qquad\text{for every }r\in\mathcal I.
\]

Let
\[
 \mathcal D_0
 =
 W^*(e_{k,i}'':k\geq1,\ i\geq1)
 \subset\mathcal A^\infty,
\]
and set
\[
 \mathcal D
 =
 \mathcal D_0
 \overline\otimes L_\infty(\Omega_\varepsilon).
\]
Here,
$\mathcal D_0$ denotes the von Neumann subalgebra of $\mathcal A^\infty$
generated by the projections $(e_{k,i}'')_{k,i}$. Equivalently, it is
the weak-operator closure of the unital $*$-algebra generated by
these projections.
The algebra $\mathcal D_0$ is commutative. Indeed, for each fixed
$k$, the projections $(e_{k,i}'')_i$ are pairwise orthogonal. If
$k\neq\ell$, then $e_{k,i}''$ and $e_{\ell,j}''$ belong to distinct
tensor coordinates and therefore commute. Consequently, $\mathcal D$
is commutative.

Let
\[
 \mathbb E:\widetilde{\mathcal A}\longrightarrow\mathcal D
\]
be the trace-preserving conditional expectation. Applying
$\mathbb E$ to the preceding inequalities gives
\[
 -\mathbb E(b)
 \leq
 \sum_{k\geq1}
 \varepsilon_ka_k\mathbb E(A_r(\gamma_k))
 \leq
 \mathbb E(b)
\]
for every $r\in\mathcal I$. Notice that
$\mathbb E(A_r(\gamma_k))$ does not depend on the Rademacher
variables.

We are now in a commutative probability space. By
\cite[Lemma~1, p.~78]{Bur64} and Fubini's theorem,
\[
 \limsup_{k\to\infty}
 \sup_{r\in\mathcal I}
 \bigl|a_k\mathbb E(A_r(\gamma_k))\bigr|
 <\infty
\qquad\text{almost everywhere}.
\]

On the other hand, since $e_{k,i}''\in\mathcal D$ and
$\mathbb E$ is $\mathcal D$-bimodular, applying $\mathbb E$ to
\eqref{est7} gives
\begin{equation}\label{est9}
 e_{k,i}''
 \mathbb E(A_{r_{k,i}}(\gamma_k))
 e_{k,i}''
 \geq
 \frac{\lambda_k}{a_k}e_{k,i}''.
\end{equation}
Since $\mathcal D$ is commutative, \eqref{est9} implies
\[
 \sup_{r\in\mathcal I}
 \bigl|a_k\mathbb E(A_r(\gamma_k))\bigr|
 \geq\lambda_ke_{k,i}''
 \qquad\text{for every }i.
\]
The projections $(e_{k,i}'')_i$ are pairwise orthogonal and satisfy
$e_k''=\sum_i e_{k,i}''$. Hence
\[
 \sup_{r\in\mathcal I}
 \bigl|a_k\mathbb E(A_r(\gamma_k))\bigr|
 \geq\lambda_k\sum_i e_{k,i}''
 =\lambda_ke_k''.
\]
By \eqref{est5},
\[
 \bigvee_{k\geq\ell}e_k''=1
 \qquad\text{for every }\ell\geq1.
\]
Since $\lambda_k\to\infty$, we obtain
\[
 \limsup_{k\to\infty}
 \sup_{r\in\mathcal I}
 \bigl|a_k\mathbb E(A_r(\gamma_k))\bigr|
 =\infty
 \qquad\text{almost everywhere},
\]
which is a contradiction. This proves
\eqref{Lambdap8 bound}.
\end{proof}

Proposition \ref{prop:borel-cantelli} follows from the next lemma whose proof is very similar to Lemma 1 in \cite{Saw}.

\begin{lemma}\label{lem:borel-cantelli0}
Let $\mathcal M$ and $\mathcal U$ satisfy the assumptions of
Proposition \ref{prop:borel-cantelli}. Let $p,q\in\mathcal A$ be
projections and let $\theta>1$. Then there exist $g\in G$ and
$u\in\mathcal U$ such that
\begin{equation*}
 \varphi\bigl(
 p\wedge\alpha_g(u^*qu)
 \bigr)
 \leq
 \theta\,\varphi(p)\varphi(q).
\end{equation*}
\end{lemma}

\begin{proof}
For arbitrary projections $p_1,p_2\in\mathcal A$, one has
\[
 \varphi(p_1\wedge p_2)\leq\varphi(p_1p_2).
\]
Thus it is enough to find $g\in G$ and $u\in\mathcal U$ such that
\begin{equation}\label{bc1}
 \varphi\bigl(p\,\alpha_g(u^*qu)\bigr)
 \leq
 \theta\,\varphi(p)\varphi(q).
\end{equation}

If either $p=0$ or $q=0$, the conclusion is immediate. Assume
therefore that $\varphi(p)\varphi(q)>0$. Suppose, toward a
contradiction, that \eqref{bc1} fails for every $g\in G$ and
$u\in\mathcal U$.

Let $\mathcal C$ be the closed convex hull in $L_2(\mathcal A)$ of
the set
\[
 \left\{
 \alpha_g(u^*qu):
 g\in G,\ u\in\mathcal U
 \right\}.
\]
Every $f\in\mathcal C$ satisfies
\begin{equation}\label{bc2}
 \varphi(f)=\varphi(q)
\end{equation}
and
\begin{equation}\label{bc3}
 \varphi(pf)
 \geq
 \theta\,\varphi(p)\varphi(q).
\end{equation}
Because $\mathcal U$ is a group, $\mathcal C$ is invariant under
all transformations
\[
 f\longmapsto\alpha_g(u^*fu),
 \qquad g\in G,\quad u\in\mathcal U.
\]

The set $\mathcal C$ is a nonempty closed convex subset of the
Hilbert space $L_2(\mathcal A)$. Hence it contains a unique element
$s$ of minimal $L_2$ norm. For every $g\in G$ and $u\in\mathcal U$,
the element $\alpha_g(u^*su)$ also belongs to $\mathcal C$ and has
the same $L_2$ norm as $s$. By uniqueness,
\begin{equation}\label{bc4}
 s=\alpha_g(u^*su)
 \qquad
 \text{for all }g\in G,\ u\in\mathcal U.
\end{equation}
Taking $u=1$ in \eqref{bc4}, we obtain
\[
 \alpha_g(s)=s
 \qquad\text{for every }g\in G.
\]
The fixed-point algebra of the translation action on
\[
 \mathcal A=L_\infty(G)\overline\otimes\mathcal M
\]
is $1\otimes\mathcal M$, which we identify with $\mathcal M$.
Therefore, $s\in\mathcal M$.
Taking $g=e$ in \eqref{bc4}, we further obtain
\[
 u^*su=s
 \qquad\text{for every }u\in\mathcal U.
\]
Equivalently, $us=su$ for every $u\in\mathcal U$, and hence
\[
 s\in\mathcal U'\cap\mathcal M=\mathbb C1.
\]
Thus $s=c1$ for some scalar $c$.

By \eqref{bc2},
\[
 c=\varphi(q).
\]
Consequently,
\[
 \varphi(ps)=\varphi(p)\varphi(q),
\]
contradicting \eqref{bc3}, since $\theta>1$ and
$\varphi(p)\varphi(q)>0$. This proves the lemma.
\end{proof}

\begin{proof}[Proof of Proposition \ref{prop:borel-cantelli}]
Put
\[
 p_k=e_k^\perp,
 \qquad k\geq1.
\]
Since
\[
 \sum_{k\geq1}\varphi(e_k)=\infty,
\]
then
\[
 \prod_{k\geq m}\varphi(p_k)
 =
 \prod_{k\geq m}\bigl(1-\varphi(e_k)\bigr)
 =0
 \qquad\text{for every }m\geq1.
\]

We divide $\mathbb N$ into successive finite intervals
\[
 I_j=\{N_j,N_j+1,\ldots,N_{j+1}-1\},
 \qquad j\geq1,
\]
chosen so that
\begin{equation*}
 \prod_{k\in I_j}\varphi(p_k)\leq2^{-j-1}.
\end{equation*}
For each $j$, choose numbers $\theta_{j,k}>1$, $k\in I_j$, such
that
\[
 \prod_{k\in I_j}\theta_{j,k}\leq2.
\]

Fix a block $I_j$. Applying Lemma
\ref{lem:borel-cantelli0} successively within this finite block, we
can choose $g_k\in G$ and $u_k\in\mathcal U$, $k\in I_j$, such that
\begin{equation}\label{bc:block-meet}
 \varphi\left(
 \bigwedge_{k\in I_j}
 \alpha_{g_k}(u_k^*p_ku_k)
 \right)
 \leq
 \left(\prod_{k\in I_j}\theta_{j,k}\right)
 \left(\prod_{k\in I_j}\varphi(p_k)\right)
 \leq2^{-j}.
\end{equation}
Indeed, after the first $m$ projections have been chosen, apply
Lemma \ref{lem:borel-cantelli0} with $p$ equal to their meet and
$q$ equal to the next projection.

For $\ell\geq1$, set
\[
 F_\ell
 :=
 \bigwedge_{k\geq\ell}
 \alpha_{g_k}(u_k^*p_ku_k).
\]
For every block $I_j$ contained in $\{k:k\geq\ell\}$, one has
\[
 F_\ell
 \leq
 \bigwedge_{k\in I_j}
 \alpha_{g_k}(u_k^*p_ku_k).
\]
It follows from \eqref{bc:block-meet} that
\[
 \varphi(F_\ell)\leq2^{-j}
\]
for arbitrarily large $j$. Hence $\varphi(F_\ell)=0$, and the
faithfulness of $\varphi$ gives $F_\ell=0$.\\
Taking orthogonal complements, we obtain
\[
 \bigvee_{k\geq\ell}
 \alpha_{g_k}(u_k^*e_ku_k)
 =1
 \qquad\text{for every }\ell\geq1.
\]
This proves Proposition \ref{prop:borel-cantelli}.
\end{proof}

\section{A combinatorial method for noncommutative product estimates}
\label{combinatorial method}
Throughout this section, $(\N,\tr)$ denotes a finite von Neumann
algebra equipped with a normal faithful tracial state; in particular,
\(
 \tr(\one)=1.
\)

Let $q>2$ be an integer and
$(a_i)_{i\in I}\subset\N_+$, put $A=\sum_{i\in I}a_i.$
Expanding $\tr(A^q)$, we write
\[
 \|A\|_q^q
 =
 \Sigma_{\mathrm{rep}}+\Sigma_{\mathrm{inj}},
\]
where
\[
 \Sigma_{\mathrm{rep}}
 =
 \sum_{\substack{
 j_1,\ldots,j_q\in I\\
 \exists\,1\leq r<s\leq q\ \text{such that }j_r=j_s
 }}
 \tr(a_{j_1}\cdots a_{j_q}),
\]
and
\[
 \Sigma_{\mathrm{inj}}
 =
 \sum_{\substack{
 j_1,\ldots,j_q\in I\\
 j_r\neq j_s\ \text{for all }1\leq r<s\leq q
 }}
 \tr(a_{j_1}\cdots a_{j_q}).
\]

By cyclicity of the trace, every tuple of distinct indices can be
cyclically rotated so that its smallest index occurs first. Hence
\begin{equation}\label{eq:cyclic-minimum}
 \Sigma_{\mathrm{inj}}
 =
 q\sum_{i\in I}\tr(a_iD_i),
\end{equation}
where
\[
 D_i
 =
 \sum_{\substack{
 j_1,\ldots,j_{q-1}>i\\
 j_r\neq j_s\ \text{for all }1\leq r<s\leq q-1
 }}
 a_{j_1}\cdots a_{j_{q-1}}.
\]
In this section, we explore a combinatorial method to estimate
$\Sigma_{\mathrm{rep}}$ and
$\left\|(D_i)_{i\in I}\right\|_
 {L_{\frac{q}{q-1}}(\N;\ell_\infty)}$.

To this end, we first introduce some combinatorial notation.
Let $\mathcal P_q$ denote the lattice of all partitions of
\(
    [q]=\{1,\ldots,q\}.
\)
We equip $\mathcal P_q$ with the following order: for
$\sigma,\pi\in\mathcal P_q$, we write
\(
    \sigma\le \pi
\)
if every block of $\sigma$ is contained in some block of $\pi$.
Let $\widehat 0$ and $\widehat 1$ denote respectively the minimal
and maximal elements of $\mathcal P_q$. Thus,
\(
    \widehat 0
    =
    \bigl\{
        \{1\},\ldots,\{q\}
    \bigr\}
\)
is the partition into $q$ singletons, whereas
\(
    \widehat 1
    =
    \bigl\{
        [q]
    \bigr\}
\)
is the partition consisting of a single block. For
$\pi\in\mathcal P_q$, let $r_1(\pi)$ denote the number of singleton
blocks of $\pi$. For a multi-index
$\mathbf i=(i_1,\ldots,i_q)\in I^q$, which
can be understood as a map
\[
    \mathbf{i}\colon [q]\longrightarrow I,
    \qquad
    r\longmapsto i_r,
\]
let $\pi(\mathbf i)$ be the partition of $[q]$ into the nonempty
fibres of $\mathbf i$, namely,
\[
\pi(\mathbf i)
=
\left\{\mathbf{i}^{-1}(\{i\})
    \;\middle|\;
    i\in I,\ \mathbf{i}^{-1}(\{i\})\neq\varnothing
\right\}.
\]
Now define the partition sum
\[
\Phi(\pi)
=
\sum_{\substack{\mathbf i\in I^q\\
                \pi(\mathbf i)=\pi}}
\tr\bigl(a_{i_1}\cdots a_{i_q}\bigr),
\qquad
\Psi(\pi)
=
\sum_{\substack{\mathbf i\in I^q\\
                \pi(\mathbf i)\ge \pi}}
\tr\bigl(a_{i_1}\cdots a_{i_q}\bigr).
\]
We easily see
\[
\Psi(\pi)
=
\sum_{\substack{\sigma\in\mathcal P_q\\
                \sigma\ge\pi}}
\Phi(\sigma).
\]
By M\"obius inversion,
\[
\Phi(\pi)
=
\sum_{\substack{\sigma\in\mathcal P_q\\
                \sigma\ge\pi}}
\mu(\pi,\sigma)\Psi(\sigma).
\]
where $\mu(\pi,\sigma)$ is the M\"obius function of the partition
lattice $\mathcal P_q$.
For the M\"obius inversion formula and the explicit formula on the
partition lattice, see
\cite[Section~3.7 and Example~3.10.4]{Stanley12}  or
\cite{Rota64}.

We first estimate $\Sigma_{\mathrm{rep}}$ and then turn to
$\left\|(D_i)_{i\in I}\right\|_{L_{\frac{q}{q-1}}(\N;\ell_\infty)}$.

\subsection{Estimate of the repeated-index term}

\begin{proposition}\label{prop:rep}
Let $q\geq2$ be an integer and let $0\leq a_i\leq\one$ for
$i\in I$. Then
\[
    \Sigma_{\mathrm{rep}}
 =
 \sum_{\substack{
 j_1,\ldots,j_q\in I\\
 \exists\,1\leq r<s\leq q\ \text{such that }j_r=j_s
 }}
 \tr(a_{j_1}\cdots a_{j_q})
\]
satisfies
\[
    |\Sigma_{\mathrm{rep}}|
    \lesssim_q
    \left\|
        \sum_{i\in I} a_i
    \right\|_{q-1}^{q-1}.
\]
\end{proposition}

We shall use the following lemma.
\begin{lemma}\label{lem: q to q-1}
Let $d\ge2$, let $(a_i)_{i\in I}\subset L_\infty^+(\mathcal{N})$
satisfy $0\le a_i\le1$, and let
$h_1,\ldots,h_d\in L_\infty(\mathcal{N})$. Suppose that
\[
    \frac{d-1}{r}+\sum_{s=1}^d\frac1{r_s}=1,
    \qquad
    1\le r,r_1,\ldots,r_d\le\infty.
\]
Then
\[
    \left|
    \tr\left(
        \sum_{i\in I}
        a_i h_1 a_i h_2\cdots a_i h_d
    \right)
    \right|
    \le
    \left\|\sum_{i\in I}a_i\right\|_r^{d-1}
    \prod_{s=1}^d\|h_s\|_{r_s}.
\]
\end{lemma}

We postpone the proof of Lemma \ref{lem: q to q-1} and first prove
Proposition \ref{prop:rep}.

\begin{proof}[Proof of Proposition \ref{prop:rep}]
By M\"obius inversion,
\begin{equation}
    \Sigma_{\mathrm{rep}}
    =
    -\sum_{\substack{
        \pi\in\mathcal P_q\\
        \pi>\widehat{0}
    }}
    \mu(\widehat{0},\pi)\,\Psi(\pi).
    \label{Mobius for rep}
\end{equation}
Hence, it suffices to estimate \(\lvert\Psi(\pi)\rvert\) for every
\(\pi>\widehat{0}\).

To this end, denote the blocks of $\pi$ by
\[
    \pi=\{B_0,B_1,\ldots,B_m\},
\]
By the tracial property, we may assume that
\[
    B_0=\{1=b_1<b_2<\cdots<b_d\},
    \qquad d\geq 2.
\]
For each \(1\leq\alpha\leq m\), let $\bigl(\xi_i^{(\alpha)}\bigr)_{i\in I}$
be a family of independent and identically distributed random
variables, uniformly distributed over the set of
\(\lvert B_\alpha\rvert\)-th roots of unity. We also assume that the
families corresponding to different values of \(\alpha\) are
independent. By definition, for \(0\leq n\leq \lvert B_\alpha\rvert\),
\[
  \bigl|\xi_i^{(\alpha)}\bigr|=1,\qquad   \mathbb{E}\left[
        \bigl(\xi_i^{(\alpha)}\bigr)^n
    \right]
    =
    \begin{cases}
        1,
        & n=0\ \text{or}\ n=\lvert B_\alpha\rvert,\\[2mm]
        0,
        & 1\leq n<\lvert B_\alpha\rvert.
    \end{cases}
\]
Define
\[
    X_\alpha
    =
    \sum_{i\in I}\xi_i^{(\alpha)}a_i,
    \qquad 1\leq\alpha\leq m.
\]
For every \(k\in[q]\setminus B_0\), there exists a unique
\(1\leq\alpha\leq m\) such that \(k\in B_\alpha\). We denote this
index by \(\alpha_k\).

Set \(b_{d+1}=q+1\). For \(1\leq i\leq d\), define
\[
    H_i
    =
    \begin{cases}
        \displaystyle
        \prod_{r=b_i+1}^{b_{i+1}-1}X_{\alpha_r},
        & b_{i+1}>b_i+1,\\[4mm]
        1,
        & b_{i+1}=b_i+1,
    \end{cases}
\]
Expanding the products and using independence, we obtain
\[
    \mathbb{E}\,
    \tr\left(
        \sum_k
        a_kH_1a_kH_2\cdots a_kH_d
    \right)
    =
    \Psi(\pi),
\]
Indeed, let
\[
 J=[q]\setminus B_0.
\]
For every $r\in J$, introduce an index $j_r\in I$. Expanding the
definition of the operators $H_1,\ldots,H_d$, while preserving the
original order of the positions $1,\ldots,q$, gives
\begin{align*}
 \sum_{k\in I}a_kH_1a_kH_2\cdots a_kH_d
 =
 \sum_{k\in I}
 \sum_{(j_r)_{r\in J}\in I^J}
 \left(
 \prod_{r\in J}
 \xi_{j_r}^{(\alpha_r)}
 \right)
 a_{i_1}a_{i_2}\cdots a_{i_q},
\end{align*}
where the multi-index
$\mathbf i=(i_1,\ldots,i_q)$ is defined by
\[
 i_r
 =
 \begin{cases}
  k,&r\in B_0,\\
  j_r,&r\in J.
 \end{cases}
\]
Thus,
\begin{align}
 &\mathbb E\,
 \tr\left(
 \sum_{k\in I}a_kH_1a_kH_2\cdots a_kH_d
 \right)
 =
 \sum_{k\in I}
 \sum_{(j_r)_{r\in J}\in I^J}
 \mathbb E\left(
 \prod_{r\in J}
 \xi_{j_r}^{(\alpha_r)}
 \right)
 \tr(a_{i_1}\cdots a_{i_q}).
 \label{eq:expectation-expansion}
\end{align}

We now compute the scalar expectation. Since the random families
corresponding to distinct blocks $B_\alpha$ are independent,
\[
 \mathbb E\left(
 \prod_{r\in J}\xi_{j_r}^{(\alpha_r)}
 \right)
 =
 \prod_{\alpha=1}^m
 \mathbb E\left(
 \prod_{r\in B_\alpha}
 \xi_{j_r}^{(\alpha)}
 \right).
\]
For every $t\in I$, define
\[
 n_{\alpha,t}
 =
 \#\{r\in B_\alpha:j_r=t\}.
\]
Then
\[
 \sum_{t\in I}n_{\alpha,t}=|B_\alpha|
\]
and, by independence of the random variables
$(\xi_t^{(\alpha)})_{t\in I}$,
\begin{align*}
 \mathbb E\left(
 \prod_{r\in B_\alpha}\xi_{j_r}^{(\alpha)}
 \right)=
 \prod_{t\in I}
 \mathbb E\left[
 \bigl(\xi_t^{(\alpha)}\bigr)^{n_{\alpha,t}}
 \right].
\end{align*}
It follows that
\[
 \mathbb E\left(
 \prod_{r\in B_\alpha}\xi_{j_r}^{(\alpha)}
 \right)
 =
 \begin{cases}
  1,&j_r=j_{r'}\text{ for all }r,r'\in B_\alpha,\\
  0,&\text{otherwise}.
 \end{cases}
\]
Consequently,
\[
 \mathbb E\left(
 \prod_{r\in J}\xi_{j_r}^{(\alpha_r)}
 \right)
 =
 \one_{\{
 i_r=i_{r'}\text{ whenever }r,r'
 \text{ belong to the same block of }\pi
 \}}.
\]
Therefore, \eqref{eq:expectation-expansion} becomes
\begin{align}\label{explica}
 \mathbb E\,
 \tr\left(
 \sum_{k\in I}a_kH_1a_kH_2\cdots a_kH_d
 \right)
 &=
 \sum_{\substack{\mathbf i\in I^q\\
 i_r=i_{r'}\text{ whenever }r,r'
 \text{ lie in the same block of }\pi}}
 \tr(a_{i_1}\cdots a_{i_q})\\\nonumber
 &\quad=
 \sum_{\substack{\mathbf i\in I^q\\
 \pi(\mathbf i)\geq\pi}}
 \tr(a_{i_1}\cdots a_{i_q})=\Psi(\pi).
\end{align}

On the other hand, put $A=\sum_{i\in I}a_i$. Since, we clearly have that
$\left[\begin{array}{cc} A & X_\alpha \\ X_\alpha ^*& A\end{array}\right]\geq 0$ as it  is a sum of positive elements.
We deduce that $\|X_\alpha\|_{q-1}\leq \|A\|_{q-1}$.

For \(1\le i\le d\), let $n_i:=b_{i+1}-b_i-1$ and define
\[
    p_i
    :=
    \begin{cases}
        \dfrac{q-1}{n_i}, & n_i\ge 1,\\[2mm]
        \infty, & n_i=0.
    \end{cases}
\]
By H\"older's inequality,
\[
\begin{aligned}
    \|H_i\|_{p_i}
    &\le
    \prod_{r=b_i+1}^{b_{i+1}-1}
    \left\|
        X_{\alpha_r}
    \right\|_{q-1}\leq
    \|A\|_{q-1}^{n_i}.
\end{aligned}
\]
When $n_i=0$, this means simply that $H_i=\one$ and
$\|H_i\|_\infty=1$. Moreover,
\[
    \sum_{i=1}^d n_i=q-d,
    \qquad
    \frac{d-1}{q-1}+\sum_{i=1}^d\frac{1}{p_i}=1.
\]
Lemma \ref{lem: q to q-1}, applied with $r=q-1$, therefore gives
\[
\begin{aligned}
    \left|
    \tr\left(
        \sum_k
        a_kH_1a_kH_2\cdots a_kH_d
    \right)
    \right|
    &\le
    \|A\|_{q-1}^{d-1}
    \prod_{i=1}^d\|H_i\|_{p_i}\leq
    \|A\|_{q-1}^{q-1}.
\end{aligned}
\]
Taking expectation, we conclude that
\[
\begin{aligned}
    |\Psi(\pi)|
    \le
    \mathbb{E}
    \left|
    \tr\left(
        \sum_k
        a_kH_1a_kH_2\cdots a_kH_d
    \right)
    \right|\lesssim_q
    \|A\|_{q-1}^{q-1}.
\end{aligned}
\]
Combining this estimate with \eqref{Mobius for rep} and using that
$\mathcal P_q$ is finite, we obtain
\[
    |\Sigma_{\mathrm{rep}}|
    \lesssim_q
    \left\|
        \sum_{i\in I} a_i
    \right\|_{q-1}^{q-1}.
\]
\end{proof}

We now prove Lemma \ref{lem: q to q-1}, starting with the case
$d=2$.
\begin{proof}[Proof of Lemma \ref{lem: q to q-1}]
Write $f_k=a_k$. We divide the argument into three steps. In the
first two steps, $d=2$. For arbitrary scalars
$\gamma_k\in\mathbb T$, we prove the slightly stronger estimate
\begin{equation}\label{holder 2}
 \left|\tr\left(\sum_k\gamma_kf_kh_1f_kh_2\right)\right|
 \leq
 \left\|\sum_kf_k\right\|_r
 \|h_1\|_{r_1}\|h_2\|_{r_2}.
\end{equation}

\emph{Step 1: the self-adjoint case.}
Assume first that $h_1$ and $h_2$ are self-adjoint. Let
$\mathcal M_i$ be the von Neumann sub-algebra generated by $h_i$, $i=1,2$, and define
\[
 \Psi_\gamma(\alpha_1,\alpha_2)
 =
 \tr\left(\sum_k\gamma_kf_k\alpha_1f_k\alpha_2\right),
 \qquad \alpha_i\in\mathcal M_i.
\]
 By a step function of $h_i$, we mean an operator of
the form
\begin{align}\label{step function}
 \beta_i
 =
 \sum_j c_j^{(i)}
 \mathbf1_{\Delta_j^{(i)}}(h_i),
\end{align}
where the sets $\Delta_j^{(i)}\subset\mathbb R$ are pairwise
disjoint Borel sets and $c_j^{(i)}\in\mathbb C$.
If $\alpha_i$ is a step function of $h_i$ of form \eqref{step function}, then
\begin{align*}
 \Psi_\gamma(\alpha_1,\alpha_2)
 =
 \sum_{k,j,\ell}
 \gamma_kc_j^{(1)}c_\ell^{(2)}
 \tr\left(
 f_k\mathbf1_{\Delta_j^{(1)}}(h_1)
 f_k\mathbf1_{\Delta_\ell^{(2)}}(h_2)
 \right).
\end{align*}
Moreover,
\begin{align*}
\tr\left(
 f_k\mathbf1_{\Delta_j^{(1)}}(h_1)
 f_k\mathbf1_{\Delta_\ell^{(2)}}(h_2)
 \right)
\geq0.
\end{align*}
Consequently,
\[
 |\Psi_\gamma(\alpha_1,\alpha_2)|
 \leq\Psi_1(|\alpha_1|,|\alpha_2|).
\]
By spectral approximation, it is therefore enough to consider
$\gamma_k=1$ and $h_1,h_2\geq0$.

 Choose
$\theta\in[0,1]$ such that
\[
 \frac1{r_1}=\frac{1-\theta}{r'},
 \qquad
 \frac1{r_2}=\frac{\theta}{r'}.
\]
After replacing $h_i$ by $h_i+\varepsilon\one$ and then letting
$\varepsilon\downarrow0$, we may use complex powers and consider
\[
 F(z)
 =
 \tr\left(
 \sum_k f_kh_1^{\frac{r_1(1-z)}{r'}}
 f_kh_2^{\frac{r_2z}{r'}}
 \right),
 \qquad 0\leq\Re z\leq1.
\]
The function $F$ is holomorphic in the open strip and continuous on
its boundary. By the three-lines lemma,
\[
 |F(\theta)|
 \leq
 \sup_{t\in\mathbb R}|F(it)|^{1-\theta}
 \sup_{t\in\mathbb R}|F(1+it)|^\theta.
\]
Using the preceding domination and H\"older's inequality, we obtain
\begin{align*}
 |F(it)|
 \leq
 \Psi_1\left(h_1^{r_1/r'},\one\right)
 =\sum_k\tr\left(f_k^2h_1^{r_1/r'}\right)
\leq
 \left\|\sum_kf_k\right\|_r
 \|h_1\|_{r_1}^{r_1/r'},
\end{align*}
where we used $f_k^2\leq f_k$. Similarly,
\[
 |F(1+it)|
 \leq
 \left\|\sum_kf_k\right\|_r
 \|h_2\|_{r_2}^{r_2/r'}.
\]
Since $r_1(1-\theta)/r'=r_2\theta/r'=1$, this proves
\eqref{holder 2} for self-adjoint $h_1,h_2$.

\emph{Step 2: the general case.}
Write
$h_1=|h_1^*|u_1^*$ and $h_2=u_2|h_2|$. The
Cauchy--Schwarz inequality gives
\begin{align*}
 &\left|\tr\left(\sum_k\gamma_kf_kh_1f_kh_2\right)\right|\\
 &\quad\leq
 \left[
 \tr\left(\sum_kf_k|h_1^*|f_k|h_2|\right)
 \right]^{1/2}
 \left[
 \tr\left(
 \sum_kf_ku_2|h_2|u_2^*f_ku_1|h_1^*|u_1^*
 \right)
 \right]^{1/2}.
\end{align*}
Applying Step 1 to both factors proves \eqref{holder 2} in general.

\emph{Step 3: reduction of the general case to $d=2$.}
It remains to consider $d\geq3$.
Let $(\gamma_k)_k$ be independent random variables uniformly
distributed over the {$(d-1)$}-th roots of unity. Then
\[
 \mathbb E(\gamma_k^\ell)=0
 \quad(1\leq\ell\leq d-2),
 \qquad
 \mathbb E(\gamma_k^{d-1})=1.
\]
Independence gives
\begin{align*}
 \sum_k f_kh_1f_kh_2\cdots f_kh_d=
 \mathbb E\left[
 \left(\sum_k\gamma_kf_kh_1f_kh_2\right)
 \prod_{i=3}^d\left(\sum_j\gamma_jf_j\right)h_i
 \right].
\end{align*}
For every choice of the scalars $\gamma_j$, the positive
factorization used above gives
\[
 \left\|\sum_j\gamma_jf_j\right\|_r
 \leq\left\|\sum_jf_j\right\|_r.
\]
Define $s$ by
\[
 \frac1s
 =
 \frac1{r_2}
 +\frac{d-2}{r}
 +\sum_{i=3}^d\frac1{r_i}.
\]
The hypothesis of the lemma gives
\[
 \frac1r+\frac1{r_1}+\frac1s=1.
\]
Thus we may apply \eqref{holder 2} with
\[
 h_2\left(\sum_j\gamma_jf_j\right)h_3
 \cdots
 \left(\sum_j\gamma_jf_j\right)h_d
\]
as its second operator. H\"older's inequality then yields
\[
 \left|
 \tr\left(\sum_k f_kh_1f_kh_2\cdots f_kh_d\right)
 \right|
 \leq
 \left\|\sum_kf_k\right\|_r^{d-1}
 \prod_{i=1}^d\|h_i\|_{r_i},
\]
which is the desired conclusion.
\end{proof}

\subsection{A maximal estimate for
\texorpdfstring{$(D_i)_{i\in I}$}{the family (Di)}}\label{subsec:maximal}

We now estimate the family $(D_i)_{i\in I}$ arising from the
decomposition of the distinct-index term in
\eqref{eq:cyclic-minimum}. Our goal is to establish a uniform
$L_{\frac{q}{q-1}}(\N;\ell_\infty)$ estimate for this family.

\begin{proposition}\label{prop:inj term}
Let $q>2$ be an integer.
 Then
\begin{equation*}
 \left\|(D_i)_{i\in I}\right\|_
 {L_{\frac{q}{q-1}}(\N;\ell_\infty)}
 \lesssim_q
 \left\|\sum_{i\in I}a_i\right\|_q^{q-1}.
\end{equation*}
\end{proposition}

We shall use the following surprizing maximal inequality:
\begin{proposition}
\label{prop:Multi-Doob}
Consider $\ell \geq 2$ and indices with $p>1$
\[
    \frac 1 2< p_1,\ldots,p_\ell \leq\infty, \quad  \frac1p=\sum_{k=1}^\ell\frac1{p_k}.\]
  For ${1\leq i\leq N, 1\leq k\leq \ell}$, consider $a_{i}^{(k)}\in L_{p_k}^+(\mathcal N)$
  and $\xi_{i}^{(k)}\in \mathbb C$ with $|\xi_{i}^{(k)}|\leq 1$. Define
  $X_{i}^{(k)}=\sum_{j>i} \xi_{j}^{(k)}a_{j}^{(k)}$.
Then
\[
    \left\|
        \left(
            X_i^{(1)}X_i^{(2)}\cdots X_i^{(\ell)}
        \right)_i
    \right\|_{L_p(\N;\ell_\infty)}
    \lesssim_{p_1,\ldots,p_\ell}
    \prod_{k=1}^\ell
    \left\|\sum_{i=1}^N a_i^{(k)}\right\|_{p_k}.
\]
\end{proposition}
We postpone the proof of Proposition \ref{prop:Multi-Doob} to the next section as it is of independent interest.

Next, recall that $\Pcal_{q-1}$ denotes the partition lattice of
$[q-1]$. For $\pi\in\Pcal_{q-1}$, define
\[
    X_i^\pi
    =
    \sum_{\substack{j_1,\ldots,j_{q-1}>i\\
                    \pi(\mathbf j)\geq\pi}}
    a_{j_1}\cdots a_{j_{q-1}},
\]
where $\mathbf j=(j_1,\ldots,j_{q-1})\in I^{q-1}$ ($I=[1,N]$). The M\"obius
inversion gives
\begin{equation*}
    D_i
    =
    \sum_{\pi\in\Pcal_{q-1}}
    \mu(\widehat 0,\pi)X_i^\pi.
\end{equation*}
Consequently,
\begin{align*}
    \left\|(D_i)_{i\in I}\right\|_
    {L_{\frac{q}{q-1}}(\N;\ell_\infty)}
    &\leq
    \sum_{\pi\in\Pcal_{q-1}}
    |\mu(\widehat 0,\pi)|
    \left\|(X_i^\pi)_{i\in I}\right\|_
    {L_{\frac{q}{q-1}}(\N;\ell_\infty)}.
\end{align*}
Since $\Pcal_{q-1}$ is finite and depends only on $q$, Proposition
\ref{prop:inj term} follows from the following lemma.

\begin{lemma}
  Let $q>2$, let $I=[1,N]$  ($N$ arbitrary).
Then, for every $\pi\in\mathcal P_{q-1}$,
\[
    \left\|
        (X_i^\pi)_{i\in I}
    \right\|_
    {L_{\frac{q}{q-1}}(\N;\ell_\infty)}
    \lesssim_{q}
    \left\|
        \sum_{j=1}^Na_j
    \right\|_q^{q-1}
\]
where the implicit constant depends only on $q$ and is independent of $\pi$, $I$,
and the family $(a_j)_{j\in I}$.
\end{lemma}
\begin{proof}
As before, denote the blocks of $\pi$ by
\[
    \pi=\{B_1,\ldots,B_m\},
\]
Let $\bigl(\xi_i^{(\alpha)}\bigr)_{i\in I}$
be independent families  of independent and identically distributed random
variables, uniformly distributed over the set of
\(\lvert B_\alpha\rvert\)-th roots of unity ($1\leq \alpha\leq m$). Define
\[
    X^{(k)}_i
    =
    \sum_{j>i}\xi_j^{(\pi(k))}a_j,
    \qquad 1\leq k \leq q-1,
  \]
  where $\pi(k)$ is the block so that $k\in B_{\pi(k)}$.
  We already explained in \eqref{explica} that
$$\mathbb E \Big(X_i^{(1)}... X_i^{(q-1)}\Big)=
 \sum_{\substack{\mathbf j\in [i+1,N]^{q-1}\\
 \pi(\mathbf j)\geq\pi}}
 a_{j_1}\cdots a_{j_{q-1}}=X^{\pi}_i.$$
 The conclusion follows directly from Proposition \ref{prop:Multi-Doob} and the fact that $L_p(\mathcal N;\ell_\infty)$ is a normed space.
\end{proof}

\subsection{A multilinear noncommutative Doob maximal inequality}\label{subsec:multilinear}

It remains to prove Proposition \ref{prop:Multi-Doob}.
\begin{lemma}\label{lem:multi-doob}
Let $(\A_n,\E_n)_n$ be an increasing filtration in a normal
semifinite tracial von Neumann algebra $(\A,\varphi)$. Let
\[
    a_n\in L_r(\A_n),
    \qquad
    x^{(j)}\in L_{p_j}(\A),
    \qquad
    x_n^{(j)}=\E_n(x^{(j)}).
\]
Then
\[
    \left\|
        \bigl(a_nx_n^{(1)}\cdots x_n^{(k)}\bigr)_n
    \right\|_{L_q(\A;\ell_\infty^c)}
    {\lesssim_{r,p_1,\dotsm,p_k}}
    \left\|(a_n)_n\right\|_{L_r(\A;\ell_\infty^c)}
    \prod_{j=1}^k\|x^{(j)}\|_{p_j},
\]
where
\[
    \frac1q=\frac1r+\sum_{j=1}^k\frac1{p_j},
    \qquad q>2,
    \qquad 2< r,p_1,\ldots,p_k\leq\infty.
\]
\end{lemma}
\begin{proof}
For $k=1$,
write $x=x^{(1)}$ and $p=p_1$, and fix $\varepsilon>0$. By the order characterization of
$L_{r/2}(\A;\ell_\infty)$, there exists
$A_{\varepsilon}\in L_{r/2}(\A)_+$ such that
\[
    a_n^*a_n\leq A_{\varepsilon}
    \qquad\text{for every }n,
\]
and
\[
    \|A_{\varepsilon}\|_{r/2}^{1/2}
    \leq
    \left\|(a_n)_n\right\|_{L_r(\A;\ell_\infty^c)}
    +\varepsilon.
\]
Since $a_n\in\A_n$, the bimodule property and Kadison's inequality give
\[
    x_n^*a_n^*a_nx_n
    =
    \E_n(a_nx)^*\E_n(a_nx)
    \leq
    \E_n(x^*a_n^*a_nx)
    \leq
    \E_n(x^*A_{\varepsilon}x).
\]
Therefore by the noncommutative Doob maximal inequality,
\begin{align*}
    \left\|
        (a_nx_n)_n
    \right\|_{L_q(\A;\ell_\infty^c)}^2
    \leq
    \left\|
        \bigl(\E_n(x^*A_{\varepsilon}x)\bigr)_n
    \right\|_{L_{q/2}(\A;\ell_\infty)}
    \lesssim_q
    \|x^*A_{\varepsilon}x\|_{q/2}
   \leq
    \|A_{\varepsilon}\|_{r/2}\|x\|_p^2.
\end{align*}
Taking square roots and letting $\varepsilon\to0$, we obtain
\[
    \left\|
        (a_nx_n)_n
    \right\|_{L_q(\A;\ell_\infty^c)}
    \lesssim_q
    \left\|(a_n)_n\right\|_{L_r(\A;\ell_\infty^c)}
    \|x\|_p.
\]
The general case follows immediately by iteration. Indeed, let
\[
 \frac1{q_j}
 =
 \frac1r+\sum_{\ell=1}^j\frac1{p_\ell},
 \qquad
 b_n^{(j)}
 =
 a_nx_n^{(1)}\cdots x_n^{(j)}.
\]
Since $q_j\geq q>2$, repeatedly applying the case $k=1$ to
$b_n^{(j-1)}x_n^{(j)}$ yields the desired estimate.
\end{proof}
\begin{remark}
  The constant in the previous Lemma is the product of the constants $D_{q_1/2}...D_{q_k/2}$
of the noncommutative Doob inequalities for $1/q_l=1/r+1/p_1+...+1/p_l$.
\end{remark}
Taking adjoints gives the corresponding row estimate. More precisely,
apply the column estimate to the adapted sequence $(a_n^*)_n$ and to
the operators
$ x^{(k)*},x^{(k-1)*},\ldots,x^{(1)*}.$
\\ Since $\E_n(x^{(j)*})=x_n^{(j)*}$ and
$ \left\|(a_n^*)_n\right\|_{L_r(\A;\ell_\infty^c)}
 =
 \left\|(a_n)_n\right\|_{L_r(\A;\ell_\infty^r)}$, 
we obtain
\[
 \left\|
  \bigl(x_n^{(1)}\cdots x_n^{(k)}a_n\bigr)_n
 \right\|_{L_q(\A;\ell_\infty^r)}
 \lesssim_{r,p_1,\dotsm,p_k}
 \left\|(a_n)_n\right\|_{L_r(\A;\ell_\infty^r)}
 \prod_{j=1}^k\|x^{(j)}\|_{p_j}.
\]

We now use Lemma \ref{lem:multi-doob} to prove
\begin{proposition}
\label{prop:Multi-Doob-mart}
   Let $(\A_n,\E_n)_n$ be an increasing filtration in a normal semifinite
tracial von Neumann algebra $(\A,\varphi)$. Let $d\geq 2$ and $1<p<\infty$ with
\[
    \frac1p=\sum_{r=1}^d\frac1{p_r}, \quad  1<p_1,\ldots,p_d\leq\infty.
\]
For each \(1\leq r\leq d\), let
\[
    x^{(r)}\in L_{p_r}(\A),
    \qquad
    x_n^{(r)}=\E_n\bigl(x^{(r)}\bigr).
\]
Then
\[
    \left\|
        \left(
            x_n^{(1)}x_n^{(2)}\cdots x_n^{(d)}
        \right)_n
    \right\|_{L_p(\A;\ell_\infty)}
    \lesssim_{p_1,\ldots,p_d}
    \prod_{r=1}^d
    \left\|x^{(r)}\right\|_{p_r}.
\]
\end{proposition}
\begin{proof}
Choose $1\leq m\leq d$ and $\theta\in[0,1]$ such that
\[
 \frac1{2p}
 =
 \sum_{j=1}^{m-1}\frac1{p_j}
 +
 \frac{\theta}{p_m}
 =
 \frac{1-\theta}{p_m}
 +
 \sum_{j=m+1}^d\frac1{p_j}.
\]
Assume first that $0<\theta<1$. By loosing at most a  factor 4, we can assume that $x^{(m)}\geq 0$. Note also that $p_j\leq 2$ for at most  one single $p_j$ and if this occurs then $j=m$.

Put \[ b_n = x_n^{(1)}\cdots x_n^{(m-1)} {x_n^{(m)}}^\theta, \qquad  c_n = {x_n^{(m)}}^{1-\theta} x_n^{(m+1)}\cdots x_n^{(d)}.\]
Then
\[
 x_n^{(1)}\cdots x_n^{(d)}=b_nc_n.
\]

We apply Lemma \ref{lem:multi-doob} to $c_n$ and $b_n^*$ for $q=2p>2$.
As $p_j>2$ for $j\neq m$, it suffices to prove that
$\|({x_n^{(m)}}^{2\eta})\|_{L_{p_m/(2\eta)}(\N;\ell_\infty)}\lesssim \|x^{(m)}\|_{p_m}^{2\eta}$ for $\eta=\theta$ or $1-\theta$.

If $\eta>1/2$, by operator convexity, we have $0\leq ({x_n^{(m)}}^{2\eta})\leq \mathbb E_n \Big({x^{(m)}}^{2\eta}\Big)$. But ${x^{(m)}}^{2\eta}\in L_{p_m/(2\eta)}(\mathcal N)$ by construction and we can use the Doob inequality for the index
$p_m/(2\eta)>1$ to conclude.

If $\eta\leq 1/2$, we first apply the Doob inequality for the index $p_m$, so
$0\leq x_n^{(m)}\leq a$ for some $a\in L_{p_m}(\mathcal N)$ with $\|a\|_{p_m}\lesssim \|x^{(m)}\|_{p_m}$. By operator monotony,
$0\leq {x_n^{(m)}}^{2\eta}\leq a^{2\eta}$ and we are done.

Therefore, we have shown
\[
 \|(b_n)_n\|_{L_{2p}(\A;\ell_\infty^r)}
 \lesssim
 \prod_{j=1}^{m-1}\|x^{(j)}\|_{p_j}
 \|x^{(m)}\|_{p_m}^{\theta}, \quad
 \|(c_n)_n\|_{L_{2p}(\A;\ell_\infty^c)}
 \lesssim
 \|x^{(m)}\|_{p_m}^{1-\theta}
 \prod_{j=m+1}^{d}\|x^{(j)}\|_{p_j}.
\]
Hence, by the factorization characterization of
$L_p(\A;\ell_\infty)$,
\begin{align*}
 \left\|
  \bigl(x_n^{(1)}\cdots x_n^{(d)}\bigr)_n
 \right\|_{L_p(\A;\ell_\infty)}
 &\leq
 \|(b_n)_n\|_{L_{2p}(\A;\ell_\infty^r)}
 \|(c_n)_n\|_{L_{2p}(\A;\ell_\infty^c)}
 \lesssim
 \prod_{j=1}^d\|x^{(j)}\|_{p_j}.
\end{align*}

If $\theta=0$ all the indices $p_j>2$ and we can  take
\[
 b_n=x_n^{(1)}\cdots x_n^{(m-1)},
 \qquad
 c_n=x_n^{(m)}\cdots x_n^{(d)}.
\]
Similarly if $\theta=1$, we take
\[
 b_n=x_n^{(1)}\cdots x_n^{(m)},
 \qquad
 c_n=x_n^{(m+1)}\cdots x_n^{(d)}.
\]
The same argument, without fractional powers, proves the corresponding
endpoint estimates.
\end{proof}
\begin{proof}[Proof of Proposition \ref{prop:Multi-Doob}.] Fix $N\geq 1$. We consider the von Neumann algebra $\A=M_{N+1}(\N)$ with its natural trace. It has a filtration given by $\A_n=M_{n+1}(\N)\oplus \N^{N-n}$ for $1\leq n\leq N$.

  We apply Proposition \ref{prop:Multi-Doob-mart} with $d=2\ell$ and the indices
  $2p_1, 2p_1, 2p_2, 2p_2, ..., 2p_\ell,2p_\ell>1$. We choose for $m=1,...,\ell$
  $$x^{(2m-1)}= \sum_{j=1}^N (a_{N-j+1}^{(m)})^{1/2}\otimes e_{1,j+1}, \qquad      x^{(2m)}=\sum_{j=1}^N \xi^{(m)}_{N-j+1}(a_{N-j+1}^{(m)})^{1/2}\otimes e_{j+1,1}.$$
  Easy computations gives that
  $\mathbb E_n x^{(2m-1)}\mathbb E_n x^{(2m)}= \sum_{j> N-n} \xi_j^{(m)}
  a_j^{(m)} \otimes e_{1,1}$ and we have that
  $\|x^{(2m)}\|_{2p_m}\|x^{(2m-1)}\|_{2p_m}=\|\sum_i a_i\|_{p_m}$. The
  conclusion of Proposition \ref{prop:Multi-Doob-mart} is exactly what
  we are looking at by noticing that for a sequence $(x_n)$ in $L_p(\N)$ its
  $L_p(\N;\ell_\infty)$-norm is the same as that of $(x_n\otimes e_{1,1})$
  in $L_p(\A;\ell_\infty)$.
\end{proof}
\begin{remark}\label{extreme case}
  Using that $x\mapsto x^\theta$ is operator monotone for $0<\theta<1$, one can easily deduce from Proposition  \ref{prop:Multi-Doob-mart} that for $a_i\in L_p^+(\mathcal N)$, and $r>0$ provided that $p>r$, we have
  \begin{align}\label{ai} \|\Big((\sum_{j>i} a_j)^r\Big)_i\|_{L_{p/r}(\N;\ell_\infty)}\lesssim_{p,r} \|\sum _i a_i\|_p.\end{align}
    It is false in general that for $r>1$, one has, for $a_i\geq 0$,
  $$\|(a_i^r)_{i\in I}\|_
    {L_p(\N;\ell_\infty)}\lesssim_{p,r} \|(a_i)_{i\in I}\|_
    {L_{rp}(\N;\ell_\infty)}^r.$$
    One can see it easily for $r=2$ by taking $\N=M_N$, $a=N e_{1,1} +Id$ and
    $I=\{b\in M_N^+ \,|\, 0\leq b\leq a\}$. If $A$ is such that
    $0\leq b^2\leq A$, then $A\geq N/4 Id$. This would give
    $N^{1+1/p}\leq_p N$.
\end{remark}

\section{Proof of the localized maximal estimate--Proposition~\ref{Key1-es}}

The purpose of this section is to prove the localized maximal
estimate in Proposition~\ref{Key1-es}. The proof follows the general
strategy of Bourgain \cite[Lemma~7.32]{Bo89}. In Subsection~\ref{subsec:proof-Key1}, the extrapolation theorem from
Section~\ref{sec:ste} and duality reduce the desired maximal estimate
to the $q$th-moment estimate \eqref{eq:q-moment-Key1}, which is then
proved using the combinatorial method from
Section~\ref{combinatorial method} and the $L_2$-estimate in
Lemma~\ref{L_2:Lem}. In Subsection~\ref{subsec:proof-L2}, we construct the approximate
Fourier multipliers needed to prove the $L_2$-estimate in
Lemma~\ref{L_2:Lem}. The Fourier-analytic part follows Bourgain's argument. By contrast,
the treatment of the moment expansion in
\eqref{eq:q-moment-Key1} presents a genuinely new difficulty.
The scalar argument is immediate, but its operator-valued analogue
is highly nontrivial because positive operators need not commute and
their products need not be positive. This
difficulty is overcome by the combinatorial method and a multilinear noncommutative Doob maximal inequality developed in
Section~\ref{combinatorial method}.

\subsection{Proof of Proposition \ref{Key1-es}}
\label{subsec:proof-Key1}
\begin{proof}
Without loss of generality, we may assume that
$N_0=2^{k_0}$ for some sufficiently large integer $k_0$. Recall from
Section~\ref{pre} that
\[
 \mathcal A_{\mathbb Z}(\mathcal N)
 =
 \ell_\infty(\mathbb Z)\overline\otimes\mathcal N,
\]
equipped with the tensor-product trace. As in
\cite[Section~4.2]{CHW24}, it suffices to prove that
\begin{equation}\label{Key1-es2}
 \left\|
 \bigl(f*K_{2^k}\bigr)_{k_0<k<2k_0}
 \right\|_{L_p(\mathcal A_{\mathbb Z}(\mathcal N);\ell_\infty)}
 \lesssim_{p,P}(\log k_0)
 \|f\|_{L_p(\mathcal A_{\mathbb Z}(\mathcal N))}
\end{equation}
for every finitely supported $\mathcal S(\mathcal N)_+$-valued function $f$ on
$\mathbb Z$. By classical cutting arguments, we can assume that $\mathcal N$ is finite and its trace is normalized.

Fix $f$ and $k_0$. Since $f$ is finitely supported, the norms in
\eqref{Key1-es2} will not change if we replace $\mathbb Z$ by a a large
finite cyclic group {\color{blue}$\mathbb Z_N$} (keeping the same formal definition
for $K$) equipped with the counting measure. By homogeneity, we can consider the normalized counting measure and it is therefore enough to prove
\begin{equation*}
 \left\|(f*K_{2^k})_{k_0<k<2k_0}\right\|_
 {L_p(\mathcal A_{\mathbb Z_N}(\mathcal N);\ell_\infty)}
 \lesssim_{p,P} (\log k_0)
 \|f\|_{L_p(\mathcal A_{\mathbb Z_N}(\mathcal N))}.
\end{equation*}

For the given target exponent $p>1$, choose an integer $q>2$
such that
\[
 p_0:=\frac{q}{q-1}<p.
\]
By Theorem \ref{pro:stein extrapolation}  followed by the Marcinkiewicz interpolation theorem \cite{JX07}
with the trivial $L_\infty$-estimate, it suffices to prove for $f\in A_{\mathbb Z_N}(\mathcal N)_+$
\begin{equation}\label{Key1-es4}
 \left\|(f*K_{2^k})_{k_0<k<2k_0}\right\|_
 {L_1(\mathcal A_{\mathbb Z_N}(\mathcal N);\ell_\infty)}
 \lesssim_{q,P}(\log k_0)
 \|f\|_{L_{p_0}(\mathcal A_{\mathbb Z_N}(\mathcal N))}.
\end{equation}
By duality, \eqref{Key1-es4} follows once we prove
\begin{equation*}
 \left\|
  \sum_{k_0<k<2k_0}g_k*K_{2^k}
 \right\|_{L_q(\mathcal A_{\mathbb Z_N}(\mathcal N))}
 \lesssim_{q,P}\log k_0
\end{equation*}
for every finite sequence $(g_k)_k$ in $A_{\mathbb Z_N}(\mathcal N)_+$ satisfying
\[
 \left\|
  \sum_{k_0<k<2k_0}g_k
 \right\|_{L_\infty(\mathcal A_{\mathbb Z_N}(\mathcal N))}
 \leq1.
\]

Let $M\simeq\log k_0$, with the implicit constant to be chosen
below, and, for $1\leq i\leq M$, set
\[
 \mathcal I_i
 :=
 \left\{
  k\in\mathbb Z:
  k_0<k<2k_0,
  \ k-k_0\equiv i\pmod M
 \right\}.
\]
Here $i=M$ represents the zero residue class. Then
\begin{align*}
 \left\|
  \sum_{k_0<k<2k_0}g_k*K_{2^k}
 \right\|_q
 &\leq
 \sum_{i=1}^M
 \left\|
  \sum_{k\in \mathcal I_i}g_k*K_{2^k}
 \right\|_q.
\end{align*}
It therefore suffices to prove independently of $i$,
\begin{equation*}
 \left\|
  \sum_{k\in \mathcal I_i}g_k*K_{2^k}
 \right\|_q
 \lesssim_{q,P} 1
\end{equation*}
provided that
\begin{equation}\label{Key1-es7}
 \left\|\sum_{k\in\mathcal I_i}g_k\right\|_\infty\leq1.
\end{equation}

Fix $i$ and to simplify, denote
\[a_k:=g_k*L_k\quad\mathrm{and}\quad L_k:=K_{2^k}, \qquad(k\in\mathcal I_i),
 \qquad
 L_{k_0}:=K_{2^{k_0}}.
\]
The original indices are retained. In particular, any two distinct
indices $k,\ell\in\mathcal I_i$ satisfy
\begin{equation}\label{eq:M-separation}
 |k-\ell|\geq M.
\end{equation}
Thus, we want to prove
\begin{equation*}
 \left\|
  \sum_{k\in\mathcal I_i}a_k  \right\|_q \lesssim_{q,P} 1.
\end{equation*}
Since each
$L_k$ is a positive probability kernel and $g_k\geq0$, we have
$a_k\geq0$. Moreover, \eqref{Key1-es7} implies
$0\leq g_k\leq\one$, and hence
\[
 0\leq a_k=g_k*L_k\leq\one.
\]
Thus, it remains to prove
\begin{equation}\label{eq:q-moment-Key1}
 \left\|\sum_{k\in\mathcal I_i}a_k\right\|_q^q
 =
 \sum_{k_1,\ldots,k_q\in\mathcal I_i}
 \tr(a_{k_1}\cdots a_{k_q})
 \lesssim_{q,P}1.
\end{equation}
Recall from Section~\ref{combinatorial method} that the moment sum
in \eqref{eq:q-moment-Key1} is decomposed as
\begin{equation*}
 \left\|\sum_{k\in\mathcal I_i}a_k\right\|_q^q
 =
 \Sigma_{\mathrm{rep}}+\Sigma_{\mathrm{inj}},
\end{equation*}
where
\[
 \Sigma_{\mathrm{rep}}
 =
 \sum_{\substack{k_1,\ldots,k_q\in\mathcal I_i\\
                   k_r=k_s\ \text{for some }1\leq r<s\leq q}}
 \tr(a_{k_1}\cdots a_{k_q})
\]
and
\[
 \Sigma_{\mathrm{inj}}
 =
 \sum_{\substack{k_1,\ldots,k_q\in\mathcal I_i\\
                   k_r\neq k_s\ \text{for all }1\leq r<s\leq q}}
 \tr(a_{k_1}\cdots a_{k_q}).
\]
For $k\in\mathcal I_i$, recall that
\begin{equation*}
 D_k
 =
 \sum_{\substack{k_2,\ldots,k_q\in\mathcal I_i\\
                  k<k_r\ \text{for all }2\leq r\leq q\\
                  k_r\neq k_s\ \text{for all }2\leq r<s\leq q}}
 a_{k_2}\cdots a_{k_q}.
\end{equation*}
By cyclicity of the trace, every ordered $q$-tuple of pairwise
distinct indices can be cyclically rotated in a unique way so that
its smallest index occurs first. Therefore,
\begin{equation}\label{eq:Sigma-inj-Key1}
 \Sigma_{\mathrm{inj}}
 =
 q\sum_{k\in\mathcal I_i}\tr(a_kD_k),
\end{equation}

The repeated-index term $\Sigma_{\mathrm{rep}}$ will be estimated by
Proposition~\ref{prop:rep}. To handle the distinct-index term
$\Sigma_{\mathrm{inj}}$, we replace the first factor $a_k$ in
\eqref{eq:Sigma-inj-Key1} by the positive element $b_k$ defined in \eqref{eq:bk-Key1} and
use Proposition~\ref{prop:inj term} to estimate the family $(D_k)_k$.
The error arising from this replacement will be controlled by the
following $L_2$-estimate.

\begin{lemma}\label{L_2:Lem}
Let $k_1,\ldots,k_q\in\mathcal I_i$ be pairwise distinct and suppose
that
\[
 k_1=\min\{k_1,\ldots,k_q\}.
\]
Then
\begin{equation*}
 \left\|
  (\widetilde L_{k_1}-\widetilde L_{k_0})
  *(a_{k_2}\cdots a_{k_q})
 \right\|_2
 \lesssim_{q,P} k_0^{-q}.
\end{equation*}
Here, $ \widetilde L$ has kernel given by $ \widetilde L(n):=L(-n).$
\end{lemma}

The proof of Lemma~\ref{L_2:Lem} is given in the next subsection.
Assuming this lemma, we continue the proof of
Proposition~\ref{Key1-es}.
For $k\in\mathcal I_i$, set
\begin{equation}\label{eq:bk-Key1}
 b_k:=g_k*L_{k_0}.
\end{equation}
By \eqref{Key1-es7},
\begin{align*}
 \sum_{k\in\mathcal I_i}b_k
 &=
 \left(\sum_{k\in\mathcal I_i}g_k\right)*L_{k_0}
 \leq
 \one*L_{k_0}
 =\one.
\end{align*}
 And, by \eqref{eq:Sigma-inj-Key1},
\begin{align}
\left|
 \Sigma_{\mathrm{inj}}
 -q\sum_{k\in\mathcal I_i}\tr(b_kD_k)
 \right|
=
 q\left|
 \sum_{\substack{k_1,\ldots,k_q\in\mathcal I_i\\
                  k_1=\min\{k_1,\ldots,k_q\}\\
                  k_1,\ldots,k_q\ \mathrm{pairwise\ distinct}}}
 \tr\left((a_{k_1}-b_{k_1})a_{k_2}\cdots a_{k_q}\right)
 \right|.
 \label{eq:structural-difference}
\end{align}
Since
\[
 a_{k_1}-b_{k_1}
 =
 g_{k_1}*(L_{k_1}-L_{k_0}),
\]
the adjointness relation for convolution gives
\begin{align*}
\tr\left((a_{k_1}-b_{k_1})a_{k_2}\cdots a_{k_q}\right)
=
 \tr\left(
  g_{k_1}
  \bigl[(\widetilde L_{k_1}-\widetilde L_{k_0})
        *(a_{k_2}\cdots a_{k_q})\bigr]
 \right).
\end{align*}
Hence, by the Cauchy--Schwarz inequality, Lemma \ref{L_2:Lem}, and
$\|g_{k_1}\|_2\leq\|g_{k_1}\|_\infty\leq1$,
\begin{align*}
\left|
 \tr\left((a_{k_1}-b_{k_1})a_{k_2}\cdots a_{k_q}\right)
 \right|
\leq
 \|g_{k_1}\|_2
 \left\|
  (\widetilde L_{k_1}-\widetilde L_{k_0})
  *(a_{k_2}\cdots a_{k_q})
 \right\|_2
 \lesssim_{q,P}k_0^{-q}.
\end{align*}
Since $|\mathcal I_i|\leq k_0$, the number of tuples in
\eqref{eq:structural-difference} is at most $k_0^q$. Consequently,
\[
 \left|
  \Sigma_{\mathrm{inj}}
  -q\sum_{k\in\mathcal I_i}\tr(b_kD_k)
 \right|
 \lesssim_{q,P} 1.
\]
In particular using the duality
pairing between $L_q(\N;\ell_1)$ and
$L_{\frac{q}{q-1}}(\N;\ell_\infty)$,
\begin{equation*}
 \Sigma_{\mathrm{inj}}
 \lesssim_{q,P}
 1
 + q\left|\sum_{k\in\mathcal I_i}\tr(b_kD_k)\right| \lesssim_{q,P} 1+ \left\|(b_k)_{k\in \mathcal I_i}\right\|_{L_q(\N;\ell_1)}
 \left\|(D_k)_{k\in \mathcal I_i}\right\|_
 {L_{\frac{q}{q-1}}(\N;\ell_\infty)}.
\end{equation*}
Put
\[
 A=\sum_{k\in\mathcal I_i}a_k.
\]Consequently, Proposition \ref{prop:inj term} implies
\[
 \Sigma_{\mathrm{inj}}
 \lesssim_{q,P}\|A\|_q^{q-1}+1.
\]
By Proposition \ref{prop:rep},
\[
 |\Sigma_{\mathrm{rep}}|
 \lesssim_q
 \|A\|_{q-1}^{q-1}
 \leq
 \|A\|_q^{q-1}.
\]
Combining the estimates for the repeated-index and distinct-index
terms, we obtain
\[
 \|A\|_q^q
 =
 \Sigma_{\mathrm{rep}}+\Sigma_{\mathrm{inj}}
 \lesssim_{q,P}
 \|A\|_q^{q-1}+1.
\]
This gives
\[
\|A\|_q
 \lesssim_{q,P} 1,
\]
which proves \eqref{eq:q-moment-Key1}.
\end{proof}

\subsection{Approximate Fourier multipliers and the
\texorpdfstring{$L_2$}{L2}-estimate}
\label{subsec:proof-L2}
In this subsection, we recall Bourgain's construction of approximate
Fourier multipliers and his Fourier-analytic estimates
\cite[Lemma~7.32]{Bo89} to prove Lemma~\ref{L_2:Lem}. The noncommutativity only plays a minor role here in \eqref{telesc} and the
reader familiar with Bourgain's arguments can skip this part.

Let $\mathbb T$ be the dual group of $\mathbb Z$, identified with
$[0,1)$, and let
\[
 K_N=\frac1N\sum_{n=1}^N\delta_{P(n)}.
\]
Its Fourier transform is
\begin{equation*}
 \widehat K_N(\alpha)
 =
 \frac1N\sum_{n=1}^N e^{-2\pi iP(n)\alpha},
 \qquad \alpha\in\mathbb T.
\end{equation*}
For $s\geq0$, define
\[
 \mathfrak R_s
 :=
 \left\{
  \xi=\frac aq\in\mathbb T\cap\mathbb Q:
  (a,q)=1,
  \ 2^s\leq q<2^{s+1}
 \right\}.
\]
For $\xi=a/q$ in reduced form, define
\begin{equation}\label{eq:S-v2}
 S(\xi)
 :=
 \frac1q\sum_{0\leq r<q}e^{-2\pi i\xi P(r)}.
\end{equation}
Let
\begin{equation}\label{w_N}
 w_N(\beta)
 :=
 \frac1N\int_0^N e^{-2\pi i\beta P(y)}\,dy
 =
 \int_0^1e^{-2\pi i\beta P(Ny)}\,dy.
\end{equation}
Choose a smooth function $\phi$ which equals one on
$[-1/10,1/10]$ and is supported in $[-1/2,1/2]$, and define
\begin{equation*}
 \widehat L_{s,N}(\alpha)
 =
 \sum_{\xi\in\mathfrak R_s}
 S(\xi)w_N(\alpha-\xi)\phi(10^s(\alpha-\xi)).
\end{equation*}

We shall use the following standard scalar estimates; see
\cite[Section 6]{Bo89}.

\begin{proposition}\label{Bo89-es1}
There exist constants $\delta_d,\varepsilon_0>0$, depending only
on $d=\deg P$, and $c_0=c_0(P)>0$ such that the following
assertions hold.
\begin{enumerate}
 \item[\rm(i)] If $\xi\in\mathfrak R_s$, then
 \begin{equation}\label{eq:Sdec-v2}
  |S(\xi)|\lesssim_{P} 2^{-s\delta_d}.
 \end{equation}

 \item[\rm(ii)] For every $\beta\in\mathbb T$,
 \begin{equation}\label{eq:wdec-v2}
  |1-w_N(\beta)|\lesssim_{P} |\beta|N^d,
  \qquad
  |w_N(\beta)|
  \lesssim_{P} (1+|\beta|N^d)^{-1/d}.
 \end{equation}

 \item[\rm(iii)] If $\xi=a/q$ is reduced,
 $q<N^{\varepsilon_0}$, and
 $|\alpha-\xi|<N^{-d+\varepsilon_0}$, then
 \begin{equation*}
  \left|
   \widehat K_N(\alpha)
   -S(\xi)w_N(\alpha-\xi)
  \right|
  \lesssim_{P} N^{-1/2}.
 \end{equation*}

 \item[\rm(iv)] For every $\alpha\in\mathbb T$,
 \begin{equation*}
  \left|
   \widehat K_N(\alpha)
   -\sum_{s=0}^\infty\widehat L_{s,N}(\alpha)
  \right|
  \lesssim_{P} N^{-c_0}.
 \end{equation*}
\end{enumerate}
\end{proposition}

We now prove Lemma~\ref{L_2:Lem}. It suffices to prove the estimate
with $L_{k_1}-L_{k_0}$ in place of
$\widetilde L_{k_1}-\widetilde L_{k_0}$, since reflection preserves
all the multiplier bounds and support properties used below. That is,
for pairwise distinct $k_1,\ldots,k_q\in\mathcal I_i$ satisfying
\[
 k_1=\min\{k_1,\ldots,k_q\},
\]
we prove
\begin{equation}\label{Claim1}
 \left\|
  (a_{k_2}\cdots a_{k_q})*(L_{k_1}-L_{k_0})
 \right\|_2
 \lesssim_{q,P} k_0^{-q}.
\end{equation}

\begin{proof}
For $1\leq r\leq q$, put
\[
 N_r:=2^{k_r},
 \qquad
 N_0:=2^{k_0}.
\]
Let $s_r$ be positive integers to be chosen below, and let $D\geq1$.
Define
\begin{equation}\label{eq:Omega-r}
 \Omega_r(\alpha)
 :=
 \sum_{0\leq s\leq s_r}
 \sum_{\xi\in\mathfrak R_s}
 S(\xi)w_{N_r}(\alpha-\xi)
 \phi(10^s(\alpha-\xi))
 \phi(D^{-1}N_r^d(\alpha-\xi)).
\end{equation}

Write $\Omega_r=\sum_{0\leq s\leq s_r}\Omega_{r,s}$. Then
\[
 \widehat L_{k_r}-\Omega_r
 =
 \left(\widehat L_{k_r}-\sum_{s\geq0}\widehat L_{s,N_r}\right)
 +\sum_{0\leq s\leq s_r}
   \bigl(\widehat L_{s,N_r}-\Omega_{r,s}\bigr)
 +\sum_{s>s_r}\widehat L_{s,N_r}.
\]
By Proposition \ref{Bo89-es1}, the first term is bounded by
$CN_r^{-c_0}$. On the support of the cutoff difference, we must have
$N_r^d|\alpha-\xi|\gtrsim D$, so \eqref{eq:wdec-v2} and the
disjointness of the major arcs in each layer give
\[
 \|\widehat L_{s,N_r}-\Omega_{r,s}\|_\infty
 \lesssim_P D^{-1/d}.
\]
Moreover, \eqref{eq:Sdec-v2} gives
\[
 \|\widehat L_{s,N_r}\|_\infty\lesssim_P2^{-s\delta_d}.
\]
Consequently,
\begin{align*}
 \|\widehat L_{k_r}-\Omega_r\|_{L_\infty(\mathbb T)}
 &\lesssim_P
 N_r^{-c_0}
 +\sum_{s\leq s_r}D^{-1/d}
 +\sum_{s>s_r}2^{-s\delta_d} \notag\\
 &\lesssim_P
 N_r^{-c_0}
 +s_rD^{-1/d}
 +2^{-s_r\delta_d}.
\end{align*}
Moreover, one has
\begin{equation*}
 \|\mathcal F^{-1}(\Omega_r)\|_{\ell_1}
 \lesssim_P4^{s_r};
\end{equation*}
see again \cite[Section 6]{Bo89}. Consequently, by Young's
inequality,
\begin{equation}\label{Claim1-es1}
 \| (\Omega_r\widehat g_{k_r})^\vee\|_\infty
 \leq
 \|g_{k_r}\|_\infty
 \|\mathcal F^{-1}(\Omega_r)\|_{\ell_1}
 \lesssim_P4^{s_r}.
\end{equation}
Moreover, by the Plancherel theorem and
$\|g_{k_r}\|_2\leq1$,
\begin{align}
 \|g_{k_r}*L_{k_r}-(\Omega_r\widehat g_{k_r})^\vee\|_2
 &\leq
 \|\widehat L_{k_r}-\Omega_r\|_\infty
 \|g_{k_r}\|_2
 \notag\\
 &\lesssim_P
 N_r^{-c_0}
 +s_rD^{-1/d}
 +2^{-s_r\delta_d}.
 \label{Claim1-es2}
\end{align}

Put
\[
 B_r:=(\Omega_r\widehat g_{k_r})^\vee,
 \qquad
 H:=L_{k_1}-L_{k_0}.
\]

Here is the place where non commutativity is used. We have a telescoping
identity
\begin{align}\label{telesc}
 a_{k_2}\cdots a_{k_q}
 ={}&
 \sum_{r=2}^q
 X_2^{(r)}\cdots X_{r-1}^{(r)}
 (a_{k_r}-B_r)
 X_{r+1}^{(r)}\cdots X_{q}^{(r)}
 +B_2\cdots B_q
\end{align}
where $X_l^{(r)}=a_{k_l}$ if $k_l>k_r$ and $X_l^{(r)}=B_l$ if
$k_l<k_r$. Indeed, if $k_{\sigma(2)}<\cdots<k_{\sigma(q)}$ for some
permutation $\sigma$, one starts by replacing $a_{k_{\sigma(2)}}$ by
$B_{\sigma(2)}$, then $a_{k_{\sigma(3)}}$ by
$B_{\sigma(3)}$ in the difference and so on. Since by Young's inequality,
$\|H\|_{\ell_1}\leq2$, and
$\|a_{k_j}\|_\infty\leq1$
\begin{align*}
&
 \left\|
 \bigl[
 X_2^{(r)}\cdots X_{r-1}^{(r)}
 (a_{k_r}-B_r)
 X_{r+1}^{(r)}\cdots X_{q}^{(r)}
 \bigr]*H
 \right\|_2
\\
&\quad\leq
 \|H\|_{\ell_1}
 \left(\prod_{k_j<k_r}\|B_{j}\|_\infty\right)
 \|a_{k_r}-B_r\|_2
 \left(\prod_{k_j>k_r}\|a_{k_j}\|_\infty\right)
\\
&\quad\leq
 2\left(\prod_{k_j<k_r}\|B_j\|_\infty\right)
 \|a_{k_r}-B_r\|_2.
\end{align*}
Therefore,
combining \eqref{Claim1-es1} and \eqref{Claim1-es2}, we obtain
\begin{align*}
&\left\|
 (a_{k_2}\cdots a_{k_q})*(L_{k_1}-L_{k_0})
 \right\|_2
\notag\\
&\quad\lesssim_P
 \sum_{r=2}^q
 4^{s_2+\cdots+s_{r-1}}
 \left(
 N_r^{-c_0}
 +s_rD^{-1/d}
 +2^{-s_r\delta_d}
 \right)
 +
 \left\|
 (B_2\cdots B_q)*(L_{k_1}-L_{k_0})
 \right\|_2.
\end{align*}

Choose constants $c_2,\ldots,c_q>0$ successively, depending only on
$q$ and $P$, sufficiently rapidly increasing, such that if we put
\[
 s_r=\lceil c_r\log k_0\rceil,
\]
then  for every $2\leq r\leq q$,
\[
 4^{s_2+\cdots+s_{r-1}}2^{-s_r\delta_d}
 \leq k_0^{-q-1}.
\]
It follows, for some $c=c(q,P)>0$, that
\begin{align*}
&\sum_{r=2}^q
 4^{s_2+\cdots+s_{r-1}}
 \left(
  N_r^{-c_0}
  +s_rD^{-1/d}
  +2^{-s_r\delta_d}
 \right)
\notag
&\lesssim_{q,P}
 k_0^{-q}
 +k_0^c\left(N_2^{-c_0}+D^{-1/d}\right).
\end{align*}

Let $\Theta$ denote the support of the Fourier transform of
$B_2\cdots B_q$. From \eqref{eq:Omega-r}, $\Theta$ is contained in a
$C_qDN_2^{-d}$-neighborhood of rational points with denominators at
most $2^{C_qs_q}$. Since $s_q\simeq_{q,P}\log k_0$, for all
sufficiently large $k_0$ we have
\[
 2^{C_qs_q}<N_0^\delta
\]
for a suitably small $\delta>0$. With the polynomial choice of $D$
made below, choose the implicit constant in $M\simeq\log k_0$
sufficiently large so that
\[
 C_qDN_2^{-d}
 <
 \min\{N_1^{-d+\delta},N_0^{-d+\delta}\},
\]
so Proposition \ref{Bo89-es1}(iii) applies to both
$\widehat L_{k_1}$ and $\widehat L_{k_0}$ on $\Theta$. Using also
\eqref{eq:wdec-v2}, we obtain
\begin{align*}
 \sup_{\alpha\in\Theta}
 |\widehat L_{k_1}(\alpha)-\widehat L_{k_0}(\alpha)|
 &\lesssim_P
 D\left(\frac{N_1}{N_2}\right)^d
 +D\left(\frac{N_0}{N_2}\right)^d
 +N_0^{-1/2}
 \notag\\
 &\lesssim_P
 D\left(\frac{N_1}{N_2}\right)^d
 +N_0^{-1/2}.
\end{align*}
Since the indices lie in the same residue class modulo $M$,
\eqref{eq:M-separation} gives
\[
 \frac{N_1}{N_2}\leq2^{-M}.
\]
Furthermore, \eqref{Claim1-es1} and the normalized trace give
\[
 \|B_2\cdots B_q\|_2
 \leq
 \prod_{r=2}^q\|B_r\|_\infty
 \lesssim_P4^{s_2+\cdots+s_q}.
\]
Therefore,
\begin{align*}
&\left\|
 (a_{k_2}\cdots a_{k_q})*(L_{k_1}-L_{k_0})
 \right\|_2
\notag\\
&\quad\lesssim_{q,P}
 k_0^{-q}
 +k_0^c\left(2^{-c_0k_0}+D^{-1/d}\right)
 +k_0^c\left(D2^{-Md}+2^{-k_0/2}\right).
\end{align*}
Choose
\[
 D=k_0^{100d(q+c)}
\]
and then choose the implicit constant in $M\simeq\log k_0$
sufficiently large. The last display is bounded by
$C_{q,P}k_0^{-q}$, which proves \eqref{Claim1}.
\end{proof}

\section{Maximal estimates for the major-arc approximations}
Throughout this section, $\mathcal M$ denotes an arbitrary semifinite
von Neumann algebra.  In accordance with the notation introduced in
Section~\ref{pre}, we write
\[
 \mathcal A_{\mathbb R}(\mathcal M)
 =L_\infty(\mathbb R)\overline\otimes\mathcal M,
 \qquad
 \mathcal A_{\mathbb Z}(\mathcal M)
 =\ell_\infty(\mathbb Z)\overline\otimes\mathcal M.
\]

The purpose of this section is to establish the maximal and
approximation estimates for the major-arc multipliers used in the
proof of the main theorem. In
Subsection~\ref{subsec:Lp-approximating}, the construction of
$M_{s,k}$ and the general strategy follow Bourgain
\cite[(7.57)--(7.59)]{Bo89}. The corresponding operator-valued
$L_p$-maximal estimate is obtained by combining this construction
with the noncommutative sampling principle
\cite[Lemma~3.5]{CHW24}.

The main new part of this section is
Subsection~\ref{subsec:L2-approximating}. Bourgain proves the scalar
counterparts of \eqref{eq:error-decay-v2} and
\eqref{eq:difference-decay-v2} by using the multi-frequency maximal
inequality \cite[Lemma~4.1]{Bo89}, whose pointwise metric-entropy
argument does not transfer directly to
$L_2(\Aalg;\ell_\infty)$. We instead combine the decomposition from
\cite[Sections~3 and~6]{Bo881} with the construction of $M_{s,k}$
from \cite[Section~7]{Bo89}. The new ingredient is Lemma~\ref{lem:comparison-v2}, which shows that
$L_{2^k}^{(\leq2^s)}$ approximates $\widehat M_{s,k}$ in the
operator-valued maximal $L_2$-norm with decay in $s$. Together with
Lemma~\ref{lem:layer-v2}, obtained through the noncommutative sampling
principle, this yields \eqref{eq:error-decay-v2} and
\eqref{eq:difference-decay-v2} without using
\cite[Lemma~4.1]{Bo89}.

\subsection{\texorpdfstring{$L_p$}{Lp}-estimates for the approximating kernels}
\label{subsec:Lp-approximating}
For $s\geq1$, set
\[
 Q_s=(2^s)!.
\]
Following Bourgain's construction
\cite[(7.57)--(7.59)]{Bo89}, for $k\geq4^s$ define the approximating
kernel $M_{s,k}$ on $\mathbb T\simeq [0,1)$ by
\begin{equation}\label{M_k}
 \widehat M_{s,k}(\theta)
 =
 \sum_{0\leq a<Q_s}
 S\left(\frac{a}{Q_s}\right)
 w_{2^k}\left(\theta-\frac{a}{Q_s}\right)
 \phi\left(
 Q_s^2\left(\theta-\frac{a}{Q_s}\right)
 \right).
\end{equation}

The main result of this subsection is the following uniform maximal
estimate.

\begin{lemma}\label{lem:M-p-v2}
For every $1<p<\infty$ and $s\geq1$,
\begin{equation*}
 \left\|
 \bigl(f*M_{s,k}\bigr)_{k\geq4^s}
 \right\|_{L_p(\Aalg;\ell_\infty)}
 \lesssim_{p,P}\|f\|_p.
\end{equation*}
\end{lemma}

The proof of Lemma~\ref{lem:M-p-v2} follows Bourgain's strategy in
\cite[Section~7]{Bo89}. In the operator-valued setting, the required
sampling step is provided by the noncommutative sampling principle
from \cite[Lemma~3.5]{CHW24}. We first establish the auxiliary
lemmas needed for the proof.

We first recall the following noncommutative sampling principle. We denote by
$T_\psi$ the Fourier multiplier on $\R$ with symbol $\psi$ and
$\mathcal T_m$ the Fourier multiplier on $\mathbb Z$ with symbol $m:\mathbb T\to \mathbb C$

\begin{lemma}\label{lem:CHW-v2}\cite[Lemma 3.5]{CHW24}
Let $\{\phi_\lambda\}_{\lambda\in I}$ be a family of bounded scalar
functions on $\mathbb R$ such that
\[
 \supp\phi_\lambda\subset(-1/2,1/2)
 \qquad\text{for every }\lambda\in I.
\]
Suppose that, for every semifinite von Neumann algebra $\mathcal M$,
\begin{equation*}
 \left\|
 \bigl(T_{\phi_\lambda}g\bigr)_{\lambda\in I}
 \right\|_
 {L_p(\mathcal A_{\mathbb R}(\mathcal M);\ell_\infty)}
\lesssim_p \|g\|_{L_p(\mathcal A_{\mathbb R}(\mathcal M))}.
\end{equation*}
For $q\geq1$ and $\lambda\in I$, define a symbol on $\mathbb T$ by
\[
 \Phi_{\lambda,q}(\xi)
 :=
 \sum_{m\in\mathbb Z}\phi_\lambda(q\xi-m).
\]
Then
\begin{equation*}
 \left\|
 \left(
  \mathcal T_{\Phi_{\lambda,q}}(f)
 \right)_{\lambda\in I}
 \right\|_
 {L_p(\mathcal A_{\mathbb Z}(\mathcal M);\ell_\infty)}
 \lesssim_p
 \|f\|_{L_p(\mathcal A_{\mathbb Z}(\mathcal M))}.
\end{equation*}
\end{lemma}

\begin{lemma}\label{lem:continuous-v2}
For $p>1$, we have the following inequality,
\begin{equation*}
  \big\|\big(T_{w_N}(f)\big)_{N\in\mathbb N}\big\|_
  {L_p(\mathcal A_{\mathbb R}(\mathcal M);\ell_\infty)}
  \lesssim_{p,P}\|f\|_{L_p(\mathcal A_{\mathbb R}(\mathcal M))},
\end{equation*}
where $w_N$ is defined in the equation \eqref{w_N} and ${\mathcal{M}}$ is an arbitrary semifinite von Neumann algebra.
\end{lemma}
\begin{proof}
  Without loss of generality, we may assume that $f\in \mathcal{S}(\mathbb{R})_+\otimes\mathcal{S}({\mathcal{M}})_+$. Then
  \begin{align*}
  T_{w_N}(f)(x)&=\int_{0}^1f(x-P(Ny))dy=\frac{1}{N}\int_{0}^Nf(x-P(y))dy.
  \end{align*}
 Recall that $P(t)=\sum_{i=1}^d b_i t^i$ with $b_d>0$.
Choose $C_P>0$ sufficiently large so that $P$ is strictly
increasing on $[C_P,\infty)$ and, for every $t\geq C_P$,
\[
 \frac12 b_dt^d\leq P(t)\leq 2b_dt^d.
\]
We also have, for every
$z\geq P(C_P)$,
\begin{equation*}
 \frac12 db_d^{1/d}z^{1-1/d}
 \leq P'\bigl(P^{-1}(z)\bigr)
 \leq 2db_d^{1/d}z^{1-1/d}.
\end{equation*}
  We set $N_P=[C_P]+1$, then we have
  \begin{align}\label{HLmax-es1}
\nonumber&\big\|\big(T_{w_N}(f)\big)_{N\in\mathbb N}\big\|_
{L_p(\mathcal A_{\mathbb R}(\mathcal M);\ell_\infty)}
\\&\lesssim \sum_{n=1}^{N_P}
\big\|T_{w_n}(f)\big\|_{L_p(\mathcal A_{\mathbb R}(\mathcal M))}
+\big\|\big(T_{w_N}(f)\big)_{N>N_P}\big\|_
{L_p(\mathcal A_{\mathbb R}(\mathcal M);\ell_\infty)}.
  \end{align}
On the one hand, for a fixed $n$, we have
\begin{equation}\label{HLmax-es2}
    \big\|T_{w_n}(f)\big\|_{L_p(\mathcal A_{\mathbb R}(\mathcal M))}
    \leq \|f\|_{L_p(\mathcal A_{\mathbb R}(\mathcal M))}.
\end{equation}
On the other hand, for $N>N_P$, we have
\begin{align*}
   T_{w_N}(f)(x)&=\frac{1}{N}\int_{0}^Nf(x-P(y))dy\\
   &\leq T_{w_{N_p}}(f)(x)+\frac{1}{N}\int_{N_P}^Nf(x-P(y))dy\\
   &=T_{w_{N_p}}(f)(x)+\frac{1}{N}\int_{P(N_P)}^{P(N)}\frac{f(x-y)}{P^{\prime}(P^{-1}(y))}dy\\
   &\lesssim_P (w_{N_P}\hat{f})^{\vee}(x)+\frac{1}{N^d}\int_{0}^{2b_dN^d}{f(x-y)}\Big({\frac{y}{N^d}\Big)}^{-1+\frac{1}{d}}dy.
  \end{align*}
  Therefore,  combining estimate \eqref{HLmax-es2} and \cite[Lemma 4.1]{CHW24}, we get
  \begin{equation}\label{HLmax-es3}
    \big\|\big(T_{w_N}(f)\big)_{N>N_P}\big\|_
    {L_p(\mathcal A_{\mathbb R}(\mathcal M);\ell_\infty)}
    \lesssim_{p,P} \|f\|_{L_p(\mathcal A_{\mathbb R}(\mathcal M))}.
  \end{equation}
We conclude by gathering estimates \eqref{HLmax-es1}, \eqref{HLmax-es2} and \eqref{HLmax-es3}.
\end{proof}

\begin{lemma}\label{lem:cyclic-v2}
Let $\varphi$ be a fixed smooth function with support in $(-1/2,1/2)$, and let
$1\le q\leq D$.  Then
\begin{align}
\left\|
\left(\mathcal T_{
\sum_{a=0}^{q-1}w_N(.-a/q)
\varphi(D(.-a/q))} f
\right)_N
\right\|_{L_p(\Aalg;\ell_\infty)}
\lesssim_{p,P}\|f\|_p.                               \label{eq:cyclic-v2}
\end{align}
\end{lemma}

\begin{proof}
Set
\[
 m_N(\xi)
 =
 w_N(\xi/q)\varphi(D\xi/q).
\]
Since $q\leq D$, we have
\[
 \supp m_N\subset(-1/2,1/2).
\]
Lemma~\ref{lem:continuous-v2} gives $$\big\|\big(T_{w_N(./q)}(f)\big)_N \big\|_{L_p(\mathcal A_{\mathbb R}(\mathcal M);\ell_\infty)} \lesssim_{p,P} \|f\|_{L_p(\mathcal A_{\mathbb R}(\mathcal M))}.$$
By Young's inequality,
\[
 \left\|(T_{m_N}h)_N\right\|_
 {L_p(\mathcal A_{\mathbb R}(\mathcal M);\ell_\infty)}
 \lesssim_{p,P}
 \|T_ {\varphi(D./q)}h\|_p\leq  \|\hat{\varphi}\|_1 \|h\|_p.
\]
Writing $n=a+\ell q$, where $0\leq a<q$ and
$\ell\in\mathbb Z$, we obtain
\begin{align*}
 \sum_{n\in\mathbb Z}m_N(q\xi-n)
 &=
 \sum_{a=0}^{q-1}\sum_{\ell\in\mathbb Z}
 w_N\left(\xi-\frac aq-\ell\right)
 \varphi\left(
 D\left(\xi-\frac aq-\ell\right)
 \right)\\
 &=
 \sum_{a=0}^{q-1}
 w_N\left(\xi-\frac aq\right)
 \varphi\left(
 D\left(\xi-\frac aq\right)
 \right).
\end{align*}
In the last identity, the functions on the right-hand side are
understood as functions on $\mathbb T$. Indeed, since
$\varphi(D\,\cdot)$ is supported in
$(-1/(2D),1/(2D))$, for each $\xi\in\mathbb T$ and each
$0\leq a<q$, at most one integer $\ell$ contributes to the inner
sum. Thus the last expression is exactly the multiplier appearing
in \eqref{eq:cyclic-v2}. Applying
Lemma~\ref{lem:CHW-v2} proves the result.
\end{proof}

\begin{lemma}\label{lem:weighted-v2}
Assume $D\ge q$.  Then
\begin{align}
\left\|
\left(
\mathcal T_{\sum_{a=0}^{q-1}S(a/q)w_N(.-a/q)
\varphi(D(.-a/q))} f
\right)_N
\right\|_{L_p(\Aalg;\ell_\infty)}
\lesssim_{p,P}\|f\|_p.                               \label{eq:weighted-v2}
\end{align}
\end{lemma}
Actually the constant also depends on the choice of $\varphi$ but this will not matter.
\begin{proof}
 Define for $\xi\in \mathbb T$
\[
 b_{q,D}(\xi)
 =
 \sum_{a=0}^{q-1}
 S(a/q)\varphi(D(\xi-a/q)).
\]
Its inverse Fourier kernel is
\begin{equation*}
 \check b_{q,D}(n)
 =
 D^{-1}\check\varphi(n/D)
 \sum_{a=0}^{q-1}S(a/q) e^{2\pi ina/q}.
\end{equation*}
Using \eqref{eq:S-v2} and the orthogonality relation
\[
 \frac1q\sum_{a=0}^{q-1} e^{2\pi iam/q}
 =
 \begin{cases}
 1,&q\mid m,\\
 0,&q\nmid m,
 \end{cases}
\]
we obtain
\begin{align*}
\sum_{a=0}^{q-1}S(a/q) e^{2\pi ina/q}
&=
\frac1q\sum_{r=0}^{q-1}
\sum_{a=0}^{q-1}
 e^{2\pi ia(n-P(r))/q}
\notag\\
&=
\#\{0\le r<q:P(r)\equiv n\pmod q\}
=:r_q(n).
\end{align*}
The function $r_q$ is $q$-periodic and clearly
\begin{align*}
 \sum_{j=0}^{q-1}r_q(j)
 & =q.
\end{align*}
Therefore, writing $n=\ell q+j$ with $0\le j<q$, we get
\begin{align*}
\|\check b_{q,D}\|_{\ell_1}
&=
\sum_{\ell\in\Z}\sum_{j=0}^{q-1}
\frac1D
\left|
\check\varphi\left(\frac{\ell q+j}{D}\right)
\right|r_q(j)\\
&\le
\sum_{\ell\in\Z}\frac qD
\max_{0\le j<q}
\left|
\check\varphi\left(\frac{\ell q+j}{D}\right)
\right|.
\end{align*}
For any $0\leq x,y\leq \frac q D$, by the mean value theorem
$$ \left|\check\varphi\left(\frac{\ell q}{D} + x\right)-\check\varphi\left(\frac{\ell q}{D} + y\right)\right|\lesssim_\varphi \min\Big\{ 1,   \frac {D}{\ell^2q}\Big\}.$$
It follows that for any integer $K>0$
$$\sum_{\ell\in\Z}\frac qD
\max_{0\le j<q}
\left|
\check\varphi\left(\frac{\ell q+j}{D}\right)
\right|\lesssim_\varphi \int_{\mathbb R} |\check\varphi| + 2\frac{Kq}D +\frac D q \sum_{|\ell|\geq K} \frac 1{\ell^2}.$$
Optimizing over $K$ gives
$$\|\check b_{q,D}\|_{\ell_1}\lesssim_\varphi 1.$$
Let $B_{q,D}$ be the convolution operator with symbol $b_{q,D}$.
 Young's inequality implies
\begin{equation}\label{eq:B-bound-v2}
 \|B_{q,D}f\|_{L_p(\Aalg)}
 \lesssim_{\varphi}\|f\|_{L_p(\Aalg)}.
\end{equation}

Choose a smooth function $\widetilde\varphi$ which equals one on
$\supp\varphi$ and is supported in $(-1/2,1/2)$. Apply Lemma \ref{lem:cyclic-v2} to $B_{q,D}f$, with
$\widetilde\varphi$ in place of $\varphi$.  The symbol of the resulting Fourier multiplier in the variable $\xi$ is
\begin{align*}
&\left[
 \sum_{a=0}^{q-1}
 w_N(\xi-a/q)
 \widetilde\varphi(D(\xi-a/q))
 \right]b_{q,D}(\xi)\\
&\quad=
\sum_{a=0}^{q-1}\sum_{b=0}^{q-1}
S(b/q)w_N(\xi-a/q)
\widetilde\varphi(D(\xi-a/q))
\varphi(D(\xi-b/q)).
\end{align*}
Note that all terms with $a\ne b$ vanish.   Since $\widetilde\varphi=1$ on $\supp\varphi$, the resulting multiplier is exactly
\[
 \sum_{a=0}^{q-1}
 S(a/q)w_N(\xi-a/q)
 \varphi(D(\xi-a/q)).
\]
Lemma \ref{lem:cyclic-v2} and \eqref{eq:B-bound-v2} now give
\eqref{eq:weighted-v2}.
\end{proof}

\begin{proof}[Proof of Lemma~\ref{lem:M-p-v2}]
For $s\ge2$, apply Lemma \ref{lem:weighted-v2} with
\[
 q=Q_s,\qquad D=Q_s^2,\qquad \varphi=\phi.
\]
The condition $D\ge q$ is then satisfied.  The case $s=1$ involves
only the fixed denominator $Q_1=2$ and follows from the same argument
with a harmless change of cutoff.
\end{proof}

\subsection{\texorpdfstring{$L_2$}{L2}-approximation estimates}
\label{subsec:L2-approximating}
The main result of this subsection is the following lemma, which
gives a summably decaying approximation of $K_{2^k}$ by $M_{s,k}$
in the operator-valued maximal $L_2$-norm, together with the
corresponding estimate for the successive differences
$M_{s,k}-M_{s-1,k}$.

\begin{lemma}\label{lem:L2decay-v2}
There is $\delta>0$ such that, for every $s\ge1$,
\begin{equation}\label{eq:error-decay-v2}
 \left\|
 \bigl(f*(K_{2^k}-M_{s,k})\bigr)_{k\ge4^s}
 \right\|_{L_2(\Aalg;\ell_\infty)}
 \lesssim2^{-\delta s}\|f\|_2.
\end{equation}
Moreover, for every $s\ge2$,
\begin{equation}\label{eq:difference-decay-v2}
 \left\|
 \bigl(f*(M_{s,k}-M_{s-1,k})\bigr)_{k\ge4^s}
 \right\|_{L_2(\Aalg;\ell_\infty)}
\lesssim2^{-\delta s}\|f\|_2.
\end{equation}
\end{lemma}

In the scalar setting, Bourgain proved the analogue of
Lemma~\ref{lem:L2decay-v2} by means of a logarithmic maximal estimate
for Fourier projections onto shrinking neighborhoods of finitely
many frequencies; see \cite[Lemma~4.1]{Bo89}.  Since a suitable
noncommutative counterpart of that estimate is not presently
available, we use a different argument, combining the decomposition
from \cite[Section~9]{Bo881} with the construction of $M_{s,k}$ in
\cite[Section~7]{Bo89}.  This gives
\eqref{eq:error-decay-v2} and \eqref{eq:difference-decay-v2} in the
operator-valued setting.  The obstruction to applying Bourgain's
argument directly is the following:
\begin{remark}\label{rem:Bourgain-Lemma41}
Let $\lambda_1,\ldots,\lambda_K\in\mathbb T$ and set
\[
 R_j
 =
 \left\{
 \xi\in\mathbb T:
 \min_{1\le r\le K}
 d_{\mathbb T}(\xi,\lambda_r)\le2^{-j}
 \right\}.
\]
Bourgain's scalar estimate \cite[Lemma~4.1]{Bo89} takes the form
\[
 \left\|
 \sup_j
 \left|
 \mathcal T_{\mathbf1_{R_j}}(f)
 \right|
 \right\|_{\ell_2(\mathbb Z)}
 \lesssim
 \bigl(\log(K+1)\bigr)^2
 \|f\|_{\ell_2(\mathbb Z)}.
\]
A natural operator-valued counterpart would be
\begin{equation}\label{eq:nc-Bourgain-Lemma41}
 \left\|
 \left(
 \mathcal T_{\mathbf1_{R_j}}(f)
 \right)_j
 \right\|_{L_2(\Aalg;\ell_\infty)}
 \lesssim
 \bigl(\log(K+1)\bigr)^2\|f\|_2.
\end{equation}
Bourgain used  pointwise supremum and a metric-entropy
argument.  As $L_2(\Aalg;\ell_\infty)$ has no analogous
pointwise description and we are not aware of any tool extending metric-entropy. It is presently unclear whether
\eqref{eq:nc-Bourgain-Lemma41} holds, and this appears to be an
interesting problem in its own right.
\end{remark}

We now use the decomposition introduced by Bourgain in
\cite[Section~9]{Bo881} and state the auxiliary estimates
needed for the proof of Lemma~\ref{lem:L2decay-v2}.

For dyadic $u=2^j$, $j\geq1$, define
\begin{equation}\label{eq:Qu-v2}
 \mathcal Q_u=(u!)^{A\lfloor\log u\rfloor},
\end{equation}
where $A=A(P)$ is sufficiently large, and put
\[
 \widetilde\Gamma_u
 =
 \left\{
 \frac{a}{q}\in\mathbb Q\cap\mathbb T:
 (a,q)=1,\ q\mid\mathcal Q_u
 \right\},
 \qquad
 \Gamma_u
 =
 \widetilde\Gamma_u\setminus\widetilde\Gamma_{u/2}.
\]

Recall that $\delta_d>0$ is the exponent in
Proposition~\ref{Bo89-es1}(i) and depends only on the degree $d$ of
$P$.  Fix parameters
\[
 0<\rho<1<\kappa<2,
 \qquad
 \varepsilon>0,
\]
with $\rho$ and $\kappa$ sufficiently close to $1$ and
$\varepsilon$ sufficiently small so that
\[
 1+\varepsilon-\rho<\delta_d.
\]
Fix $\zeta\in C_c^\infty((-2,2))$ such that
$0\leq\zeta\leq1$ and $\zeta=1$ on $[-1,1]$. For each dyadic $u$,
set
\[
 \varphi_u(\beta):=\zeta(2^{u^\rho}\beta),
 \qquad
 \vartheta_u(\beta):=\zeta(2^{u^\kappa}\beta).
\]
Then
\begin{align*}
 \varphi_u(\beta)
 &=1
 \quad\bigl(|\beta|\leq2^{-u^\rho}\bigr),
 &
 \supp\varphi_u
 &\subset
 \bigl\{|\beta|\leq2^{1-u^\rho}\bigr\},
 \\
 \vartheta_u(\beta)
 &=1
 \quad\bigl(|\beta|\leq2^{-u^\kappa}\bigr),
 &
 \supp\vartheta_u
 &\subset
 \bigl\{|\beta|\leq2^{1-u^\kappa}\bigr\}.
\end{align*}

Fix a smooth cutoff $\chi_0$ which equals one in a neighborhood of
the origin and is supported in $(-1/2,1/2)$.  Define
\begin{equation}\label{eq:LN-v2}
 L_N(\alpha)
 =
 w_N(\alpha)\chi_0(\alpha)
 +
 \sum_{u\ {\rm dyadic}}\Psi_{u,N}(\alpha).
\end{equation}
For these choices of parameters, Bourgain's estimate \cite[(9.17)]{Bo881} then gives an exponent
$\sigma>1/2$ such that
\begin{equation}\label{eq:scalar-error-v2}
 \|\widehat K_N-L_N\|_{L_\infty(\mathbb T)}
 \lesssim_{P} (\log N)^{-\sigma},
 \qquad N\geq3.
\end{equation}
We shall also use the following lemma, whose proof is given later in
the subsection.

\begin{lemma}\label{lem:layer-v2}
There is $\gamma>0$ such that, for every dyadic $u$,
\begin{equation}\label{eq:layer-v2}
 \left\|
 \bigl(\mathcal T_{\Psi_{u,N}}(f)\bigr)_N
 \right\|_{L_2(\Aalg;\ell_\infty)}
 \lesssim u^{-\gamma}\|f\|_2.
\end{equation}
\end{lemma}

For $s\ge1$, put
\begin{equation}\label{eq:defLNs}
 L_N^{(\le 2^s)}
 =w_N\chi_0+
 \sum_{\substack{u\le 2^s\\u\ {\rm dyadic}}}\Psi_{u,N}.
\end{equation}
The following estimate is the main ingredient in the
proof of Lemma~\ref{lem:L2decay-v2}.  Its proof is postponed until
the end of this subsection.

\begin{lemma}\label{lem:comparison-v2}
There is {$\gamma'>0$} such that, for every $s\ge1$,
\begin{equation}\label{eq:comparison-v2}
 \left\|
 \bigl(
\mathcal T_{
 L_{2^k}^{(\le2^s)}-\widehat M_{s,k}}( f)
 \bigr)_{k\ge4^s}
 \right\|_{L_2(\Aalg;\ell_\infty)}
 \lesssim 2^{-\gamma's}\|f\|_2.
\end{equation}
\end{lemma}

Lemma~\ref{lem:layer-v2} is obtained by combining Bourgain's
decomposition of the rational frequencies into the sets
$\Gamma_u$ with the noncommutative sampling principle.  The main
new ingredient is the estimate in
Lemma~\ref{lem:comparison-v2}, which connects the two approximations
$L_{2^k}^{(\le2^s)}$ and $M_{s,k}$ arising from
\cite{Bo881} and \cite{Bo89}, respectively.
Assuming Lemmas~\ref{lem:layer-v2} and
\ref{lem:comparison-v2}, we can now prove the main result of this
subsection.

\begin{proof}[Proof of Lemma~\ref{lem:L2decay-v2}]
For every $k\ge4^s$,
\begin{align*}
 \widehat K_{2^k}-\widehat M_{s,k}
 =(\widehat K_{2^k}-L_{2^k})
 +(L_{2^k}-L_{2^k}^{(\le2^s)})+(L_{2^k}^{(\le2^s)}-\widehat M_{s,k}).
\end{align*}
We estimate the three resulting sequences separately.

For the first one, \eqref{eq:scalar-error-v2} gives
\begin{align*}
\left\|
 \bigl(\mathcal T_{\widehat K_{2^k}-L_{2^k}}(f)
 \bigr)_{k\ge4^s}
 \right\|_{L_2(\Aalg;\ell_\infty)}
&\lesssim\left(\sum_{k\ge4^s}k^{-2\sigma}\right)^{1/2}\|f\|_2\\
&\lesssim2^{-(2\sigma-1)s}\|f\|_2.
\end{align*}
Next,
\[
 L_{2^k}-L_{2^k}^{(\le2^s)}
 =
 \sum_{\substack{u>2^s\\u\ {\rm dyadic}}}\Psi_{u,2^k}.
\]
By the triangle inequality,
Lemma \ref{lem:layer-v2} yields
\begin{align*}
&\left\|
 \bigl(\mathcal T_{L_{2^k}-L_{2^k}^{(\le2^s)}}(f)\bigr)_{k\ge4^s}
 \right\|_{L_2(\Aalg;\ell_\infty)}
 \\
&\qquad\lesssim\sum_{\substack{u>2^s\\u\ {\rm dyadic}}}u^{-\gamma}\|f\|_2
 =
 \sum_{j>s}2^{-\gamma j}\|f\|_2
 \lesssim2^{-\gamma s}\|f\|_2.
\end{align*}
The third sequence is bounded by Lemma \ref{lem:comparison-v2},
\[
 \left\|
 \bigl(\mathcal T_{L_{2^k}^{(\le2^s)}-\widehat M_{s,k}}(f)\bigr)_{k\ge4^s}
 \right\|_{L_2(\Aalg;\ell_\infty)}
 \lesssim2^{-\gamma's}\|f\|_2.
\]
Therefore, with
\[
 \delta=\min\{2\sigma-1,\gamma,\gamma'\}>0,
\]
the triangle inequality proves \eqref{eq:error-decay-v2}.

Finally, let $s\ge2$.  For every $k\ge4^s$,
\[
 \widehat M_{s,k}-\widehat M_{s-1,k}
 =
 (\widehat M_{s,k}-\widehat K_{2^k})
 +(\widehat K_{2^k}-\widehat M_{s-1,k}).
\]
Using \eqref{eq:error-decay-v2} first at level $s$ and then at level
$s-1$, and observing that
$\{k:k\ge4^s\}\subset\{k:k\ge4^{s-1}\}$, we obtain
\begin{align*}
&\left\|
 \bigl(f*(M_{s,k}-M_{s-1,k})\bigr)_{k\ge4^s}
 \right\|_{L_2(\Aalg;\ell_\infty)}
 \\
&\qquad\lesssim2^{-\delta s}\|f\|_2
 +2^{-\delta(s-1)}\|f\|_2
 \lesssim2^{-\delta s}\|f\|_2.
\end{align*}
This proves \eqref{eq:difference-decay-v2}.
\end{proof}

It remains to prove Lemmas~\ref{lem:layer-v2}
and~\ref{lem:comparison-v2}.

We collect below the scalar facts and auxiliary estimates needed for
their proofs.  They
follow from Bourgain's construction in
\cite[Sections~3 and~6]{Bo881}, the major-arc estimates in
\cite[Section~6]{Bo89}, and some elementary consequences of the
definitions above.

\begin{proposition}\label{prop:scalar-v2}
The parameters chosen above have the following properties:

\begin{enumerate}[label=(\roman*)]
\item[\rm{(i)}] If $\theta=a/q\in\Gamma_u$, then
$q\gtrsim u$ and $|S(\theta)|\lesssim u^{-\delta_d}$;

\item[\rm{(ii)}] $\mathcal Q_u|\supp\vartheta_u|\ll1$;

\item[\rm(iii)] For $0<\rho'<\rho$ sufficiently close to $\rho$, set
\[
 \gamma:=\delta_d-(1+\varepsilon-\rho')>0.
\]
Then the denominators $q\leq u^{1+\varepsilon}$ occurring in
$\Gamma_u$ can be partitioned into at most
$Cu^{1+\varepsilon-\rho'}$ blocks, and the least common multiple
$Q_\Lambda$ of each block satisfies
\[
 Q_\Lambda|\supp\varphi_u|\ll1.
\]
\end{enumerate}
\end{proposition}

\begin{proof}
Let $\theta=a/q\in\Gamma_u$, with $(a,q)=1$.  By definition,
\[
 q\mid\mathcal Q_u,
 \qquad
 q\nmid\mathcal Q_{u/2}.
\]
Thus
$q>u/2$, and \eqref{eq:Sdec-v2} gives
\[
 |S(\theta)|
 \lesssim q^{-\delta_d}
 \lesssim u^{-\delta_d}.
\]
This proves (i).

Next, Stirling's formula gives
\[
 \log\mathcal Q_u
 =
 A\lfloor\log u\rfloor\log(u!)
 \lesssim u(\log u)^2.
\]
Since $\kappa>1$ and
$|\supp\vartheta_u|\lesssim2^{-u^\kappa}$, we obtain
\[
 \mathcal Q_u|\supp\vartheta_u|
 \lesssim
 \exp\!\bigl(Cu(\log u)^2-cu^\kappa\bigr)
 \ll1,
\]
which proves (ii).

For (iii), divide the denominators
$q\le u^{1+\varepsilon}$ occurring in $\Gamma_u$ into blocks, each
containing at most $\lfloor u^{\rho'}\rfloor$ denominators.  The
number of blocks is at most
\[
 Cu^{1+\varepsilon-\rho'}.
\]
If $\Lambda$ is one such block, then
\[
 Q_\Lambda
 \le
 \prod_{q\in\Lambda}q,
\]
and therefore
\[
 \log Q_\Lambda
 \le
 Cu^{\rho'}\log u.
\]
Since $\rho'<\rho$ and
$|\supp\varphi_u|\lesssim2^{-u^\rho}$,
\[
 Q_\Lambda|\supp\varphi_u|
 \lesssim
 \exp\!\bigl(Cu^{\rho'}\log u-cu^\rho\bigr)
 \ll1.
\]

\end{proof}

\begin{lemma}\label{lem:packet-v2}
Let $Q\ge1$ and $E\subset Q^{-1}\Z\cap \mathbb T$.  Let $\chi$ be smooth and
\[
 \supp\chi\subset\left(-\frac1{2Q},\frac1{2Q}\right).
\]
For scalars $(c_\theta)_{\theta\in E}\in \ell_\infty$, consider the symbol on $\mathbb T$
$$H_N(\alpha)=\sum_{\theta\in E}c_\theta w_N(\alpha-\theta) \chi(\alpha-\theta).$$
Then, for every $f\in L_2(\Aalg)$,
\begin{equation*}
 \|(\mathcal T_{H_N}f)_N\|_{L_2(\Aalg;\ell_\infty)}
 \lesssim \sup_{\theta\in E}|c_\theta|\,\|f\|_2.
\end{equation*}
\end{lemma}
\begin{proof}By  Lemma \ref{lem:cyclic-v2} with $q=D=Q$ there, we have
$$\|(\mathcal T_{J_N}f)_N\|_{L_2(\Aalg;\ell_\infty)}
 \lesssim  \|f\|_2,$$ with $J_N(\cdot)=\sum_{\theta\in Q^{-1}\mathbb Z \cap \mathbb T} w_N(\cdot-\theta) \chi(\cdot-\theta)$.
 The intervals $(u-\frac 1{2Q},u+\frac 1{2Q})$ are pairwise disjoint
  when $u\in Q^{-1}\mathbb Z$. Thus the Fourier multiplier on $\mathbb Z$
  with symbol $\Gamma=\sum_{\theta\in E} c_\theta 1_{(\theta-\frac 1{2Q},\theta+\frac 1{2Q})}$ is bounded on $L_2(\Aalg)$ by $\sup_{\theta\in E}|c_\theta|$. Since $J_N\Gamma=H_N$ by disjointness, the Lemma is proved.
\end{proof}

We now prove Lemma~\ref{lem:layer-v2} by applying
Lemma~\ref{lem:packet-v2} separately to the two sums defining
$\Psi_{u,N}$.

\begin{proof}[Proof of Lemma~\ref{lem:layer-v2}]
Write
\[
 \Psi_{u,N}=A_{u,N}+B_{u,N},
\]
where
\begin{align*}
 A_{u,N}(\alpha)
 &=
 \sum_{\theta=a/q\in\Gamma_u}
 S(\theta)w_N(\alpha-\theta)
 \vartheta_u(\alpha-\theta),\\
 B_{u,N}(\alpha)
 &=
 \sum_{\substack{\theta=a/q\in\Gamma_u\\
                  q\le u^{1+\varepsilon}}}
 S(\theta)w_N(\alpha-\theta)
 (\varphi_u-\vartheta_u)(\alpha-\theta).
\end{align*}

For every $\theta=a/q\in\Gamma_u$, one has
$q\mid\mathcal Q_u$.  Thus all the centers occurring in
$A_{u,N}$ belong to
$\mathcal Q_u^{-1}\mathbb Z\cap\mathbb T$.  Moreover,
Proposition~\ref{prop:scalar-v2}(ii) ensures that the cutoff may be
chosen so that
\[
 \supp\vartheta_u
 \subset
 \left(
 -\frac{1}{2\mathcal Q_u},
 \frac{1}{2\mathcal Q_u}
 \right).
\]
Since Proposition~\ref{prop:scalar-v2}(i) gives
\[
 \sup_{\theta\in\Gamma_u}|S(\theta)|
 \lesssim u^{-\delta_d},
\]
Lemma~\ref{lem:packet-v2}, applied with
$Q=\mathcal Q_u$, yields
\begin{equation}\label{eq:narrow-v2}
 \left\|
 \bigl(\mathcal T_{A_{u,N}}(f)\bigr)_N
 \right\|_{L_2(\Aalg;\ell_\infty)}
\lesssim u^{-\delta_d}\|f\|_2.
\end{equation}

For $B_{u,N}$, let $\Lambda$ be one of the denominator blocks from
Proposition~\ref{prop:scalar-v2}(iii), and define
\[
 B_{u,N}^{\Lambda}(\alpha)
 =
 \sum_{\substack{\theta=a/q\in\Gamma_u\\q\in\Lambda}}
 S(\theta)w_N(\alpha-\theta)
 (\varphi_u-\vartheta_u)(\alpha-\theta).
\]
Then
\[
 B_{u,N}
 =
 \sum_{\Lambda}B_{u,N}^{\Lambda}.
\]
If $Q_\Lambda$ denotes the least common multiple of the denominators
in $\Lambda$, then all the centers occurring in
$B_{u,N}^{\Lambda}$ belong to
$Q_\Lambda^{-1}\mathbb Z\cap\mathbb T$.  Furthermore,
Proposition~\ref{prop:scalar-v2}(iii) gives
\[
 \supp(\varphi_u-\vartheta_u)
 \subset
 \left(
 -\frac{1}{2Q_\Lambda},
 \frac{1}{2Q_\Lambda}
 \right).
\]
Therefore, as before with Lemma~\ref{lem:packet-v2} gives
\[
 \left\|
 \bigl(\mathcal T_{B_{u,N}^{\Lambda}}(f)\bigr)_N
 \right\|_{L_2(\Aalg;\ell_\infty)}
 \lesssim u^{-\delta_d}\|f\|_2.
\]
There are at most
$Cu^{1+\varepsilon-\rho'}$ such blocks.  Hence the triangle
inequality yields
\begin{align}
 \left\|
 \bigl(\mathcal T_{B_{u,N}}(f)\bigr)_N
 \right\|_{L_2(\Aalg;\ell_\infty)}
 &\lesssim u^{1+\varepsilon-\rho'-\delta_d}\|f\|_2
 \notag\\
 &=u^{-\gamma}\|f\|_2.
 \label{eq:wide-v2}
\end{align}
Since $0<\gamma<\delta_d$, combining
\eqref{eq:narrow-v2} and \eqref{eq:wide-v2} proves
\eqref{eq:layer-v2}.
\end{proof}
We now prove Lemma~\ref{lem:comparison-v2}.

\begin{proof}[Proof of Lemma~\ref{lem:comparison-v2}]
For clarity, we divide the proof into two parts.  First, we estimate
the contribution of the rational centers which occur in
$L_{2^k}^{(\le2^s)}$ but not in $\widehat M_{s,k}$, using the decay
of $S(\theta)$ and Lemma~\ref{lem:packet-v2}.  We then estimate the
error caused by the different cutoffs at the common centers, using
the separation of the rational centers, the support properties of
the cutoffs, and the decay estimate \eqref{eq:wdec-v2}.

\medskip

\emph{Step 1: the contribution of the additional rational centers.}
We first identify and estimate the rational centers which occur in the
sum defining $L_{2^k}^{(\le2^s)}$ but not in $\widehat M_{s,k}$.  Set
\[
 \Theta_s
 =
 \left\{
 \frac aq\in\Q\cap\mathbb T:(a,q)=1,\ q\mid\mathcal Q_{2^s}
 \right\}, \quad\mathcal Q_{2^s}=(2^s)!^{A\lfloor\log 2^s\rfloor}
\]
and
\[
 \Lambda_s
 =
 \left\{
 \frac aq\in\Q\cap\mathbb T:(a,q)=1,\ q\mid Q_s
 \right\},
 \qquad Q_s=(2^s)!.
\]
By the definition of the sets $\Gamma_u$,
\[
 \{0\}\cup
 \bigcup_{\substack{u\le2^s\\u\ {\rm dyadic}}}\Gamma_u
 =\Theta_s.
\]
Thus $\Theta_s$ is the set of centers occurring in
$L_{2^k}^{(\le2^s)}$, whereas $\Lambda_s$ is the set of centers
occurring in $\widehat M_{s,k}$.  Since
\[
 Q_s=(2^s)!\mid\mathcal Q_{2^s},
\]
we have $\Lambda_s\subset\Theta_s$.

When $a/Q_s$ reduces to $a'/q$, the definition of the exponential
sum gives
\[
 S(a/Q_s)=S(a'/q).
\]
Consequently, the sum defining $\widehat M_{s,k}$ may be indexed by the reduced
fractions in $\Lambda_s$, with coefficient $S(\theta)$ at the center
$\theta$.  It remains in this step to estimate the centers in
$\Theta_s\setminus\Lambda_s$.

Let $\theta=a/q\in\Theta_s\setminus\Lambda_s$ be written in reduced
form.  Then
\[
 q\mid\mathcal Q_{2^s},
 \qquad
 q\nmid(2^s)!.
\]
For a prime $\ell$ and a positive integer $n$, we write
\[
 v_\ell(n)
 =
 \max\{\nu\in\mathbb N\cup\{0\}:\ell^\nu\mid n\}.
\]
Thus $v_\ell(n)$ is the exponent of $\ell$ in the prime
factorization of $n$.  We shall also use the standard formula
\[
 v_\ell(n!)
 =
 \sum_{\nu\ge1}
 \left\lfloor\frac{n}{\ell^\nu}\right\rfloor.
\]

Since $q|\mathcal Q_{2^s}$ its prime factor are smaller than $2^s$, so
if $q\nmid(2^s)!$, there is a prime $\ell\leq2^s$ such that
\[
 v_\ell(q)>v_\ell((2^s)!).
\]
We claim that this implies
\[
 4^s\lesssim q.
\]
If $\ell>2^{s-1}$, then $v_\ell((2^s)!)=1$, and therefore
$q\ge\ell^2\ge2^{2s-2}$.  If $\ell\le2^{s-1}$, put
\[
 m=\left\lfloor\frac{2^s}{\ell}\right\rfloor\ge2.
\]
Since $v_\ell((2^s)!)\ge m$, the integer $q$ is divisible by
$\ell^{m+1}$.  Moreover, $2^s<(m+1)\ell$, and hence
\[
 \frac{q}{4^s}
 \ge
 \frac{\ell^{m+1}}{4^s}
 >
 \frac{\ell^{m-1}}{(m+1)^2}
 \ge
 \frac{2^{m-1}}{(m+1)^2}
 \ge
 \frac{2}{9}.
\]
This proves the claim.

It follows from \eqref{eq:Sdec-v2} that every such center satisfies
\[
 |S(\theta)|
 \lesssim q^{-\delta_d}
 \lesssim 2^{-2\delta_ds}.
\]
Moreover, these centers can occur only in the first part of
$\Psi_{u,N}$.  Indeed, a center in its second part must satisfy
\[
 q\le u^{1+\varepsilon}
 \le2^{(1+\varepsilon)s},
\]
which is incompatible with $q\ge c2^{2s}$ because
$\varepsilon<1$.  The finitely many small values of $s$ are absorbed
into the constant.

Let
\[
 e_{s,k}(\alpha)
 =
 \sum_{\substack{u\le2^s\\u\ {\rm dyadic}}}
 \sum_{\theta\in\Gamma_u\cap
       (\Theta_s\setminus\Lambda_s)}
 S(\theta)w_{2^k}(\alpha-\theta)
 \vartheta_u(\alpha-\theta).
\]
For each dyadic $u\le2^s$, Lemma~\ref{lem:packet-v2}
applies to the terms indexed by
$\Gamma_u\cap(\Theta_s\setminus\Lambda_s)$.
There are at most $s+1$ such values of $u$.  Therefore
\[
 \left\|
 \bigl(\mathcal T_{e_{s,k}}(f)\bigr)_{k\ge4^s}
 \right\|_{L_2(\Aalg;\ell_\infty)}
 \lesssim(1+s)2^{-2\delta_ds}\|f\|_2
 \lesssim2^{-\delta_ds}\|f\|_2.
\]

\medskip

\emph{Step 2: the cutoff error at the common centers.}
For $\theta=a/q\in\Gamma_u$, let
\[
 \eta_\theta(\beta)
 =
 \begin{cases}
 \varphi_u(\beta),&q\le u^{1+\varepsilon},\\
 \vartheta_u(\beta),&q>u^{1+\varepsilon}.
 \end{cases}
\]
At the origin we put $\eta_0=\chi_0$.  Then the contribution of the
common centers to the difference of the two multipliers is
\[
 d_{s,k}(\alpha)
 =
 \sum_{\theta\in\Lambda_s}
 S(\theta)w_{2^k}(\alpha-\theta)
 \left[
 \eta_\theta(\alpha-\theta)
 -
 \phi\!\left(Q_s^2(\alpha-\theta)\right)
 \right].
\]
By the description of the two sets of centers in Step~1 and the
definition of $e_{s,k}$, this gives the exact identity
\begin{equation}\label{eq:comparison-decomposition-v2}
 L_{2^k}^{(\le2^s)}-\widehat M_{s,k}
 =e_{s,k}+d_{s,k}.
\end{equation}

We next determine the support of the cutoff difference
\[
 \eta_\theta(\beta)-\phi(Q_s^2\beta).
\]  Since $u\le2^s$, the definitions of $\varphi_u$ and
$\vartheta_u$ imply that
\[
 \eta_\theta(\beta)=1
 \qquad\text{whenever}\qquad
 |\beta|\lesssim2^{-2^{\kappa s}}.
\]
Also,
\[
 Q_s^{-2}=((2^s)!)^{-2}
 =2^{-O(s2^s)}
 \gg2^{-2^{\kappa s}},
\]
because $\kappa>1$.  Since $\phi$ equals one near zero, after
decreasing $c>0$ if necessary,
\[
 \phi(Q_s^2\beta)=1
 \qquad\text{whenever}\qquad
 |\beta|\lesssim2^{-2^{\kappa s}}.
\]
The same assertion holds for the fixed cutoff $\chi_0$ at the
origin.  Consequently,
\[
 \eta_\theta(\beta)-\phi(Q_s^2\beta)=0
 \qquad\text{if}\qquad
 |\beta|\lesssim2^{-2^{\kappa s}}.
\]
Thus the cutoff error is supported where
\[
 |\beta|\lesssim2^{-2^{\kappa s}}.
\]
For $k\ge4^s$, estimate \eqref{eq:wdec-v2} therefore gives
\[
 |w_{2^k}(\beta)|
 \lesssim2^{-k}|\beta|^{-1/d}
 \lesssim2^{-k+C_1 2^{\kappa s}}
\]
on the support of the cutoff error.

For each common center $\theta\in\Lambda_s$, define
\[
 p_{\theta,k}(\alpha)
 =
 S(\theta)w_{2^k}(\alpha-\theta)
 \left[
 \eta_\theta(\alpha-\theta)
 -
 \phi\!\left(Q_s^2(\alpha-\theta)\right)
 \right].
\]
Then
\[
 d_{s,k}=\sum_{\theta\in\Lambda_s}p_{\theta,k}.
\]
By the preceding estimate,
\[
 \|p_{\theta,k}\|_\infty
 \lesssim 2^{-k+C_1 2^{\kappa s}}.
\]
We next estimate
\[
 d_{s,k}=\sum_{\theta\in\Lambda_s}p_{\theta,k}.
\]
Since the preceding bound applies to each complete difference
$p_{\theta,k}$, it remains to control how many of these terms can be
nonzero at a fixed point $\alpha\in\Torus$.

For $\alpha\in\Torus$, set
\[
 \mathcal E_s(\alpha)
 =
 \left\{
 \theta\in\Lambda_s:
 \eta_\theta(\alpha-\theta)\neq0
 \right\}
\]
and
\[
 \mathcal F_s(\alpha)
 =
 \left\{
 \theta\in\Lambda_s:
 \phi\!\left(Q_s^2(\alpha-\theta)\right)\neq0
 \right\}.
\]
We first estimate the cardinality of $\mathcal E_s(\alpha)$.  Fix a
dyadic number $u\leq2^s$.  For the centers
$\theta=a/q\in\Gamma_u$ with $q>u^{1+\varepsilon}$, we have
$\eta_\theta=\vartheta_u$.  These centers can be written with the
common denominator $\mathcal Q_u$, and
\[
 \mathcal Q_u|\supp\vartheta_u|\ll1.
\]
Consequently, the corresponding translates of $\vartheta_u$ are
pairwise disjoint, so at most one of them can be nonzero at a given
point $\alpha$.
For the centers with $q\leq u^{1+\varepsilon}$, we have
$\eta_\theta=\varphi_u$.  By Proposition
\ref{prop:scalar-v2}(iii), these centers can be divided into at most
$Cu^{1+\varepsilon-\rho'}$ blocks.  Inside each block the centers can
be written with a common denominator $Q_\Lambda$, and
\[
 Q_\Lambda|\supp\varphi_u|\ll1.
\]
Hence the corresponding translates of $\varphi_u$ are pairwise
disjoint inside each block.  Thus, at any fixed $\alpha$, at most
$Cu^{1+\varepsilon-\rho'}$ such terms are nonzero.  Including the
single center at the origin, we obtain
\begin{align}
 \#\mathcal E_s(\alpha)
 \le
 1+
 C\sum_{\substack{u\leq2^s\\u\ {\rm dyadic}}}
 \left(1+u^{1+\varepsilon-\rho'}\right)
 \lesssim 2^{C_0s}
 \label{eq:eta-overlap-v2}
\end{align}
for some $C_0>0$, uniformly in $\alpha$.

On the other hand, the elements of $\Lambda_s$ belong to
$Q_s^{-1}\mathbb Z\cap\mathbb T$, and hence two distinct centers are
separated by at least $Q_s^{-1}$.  Since
\[
 \phi\!\left(Q_s^2(\alpha-\theta)\right)
\]
is supported in an interval of length $O(Q_s^{-2})$, these
translated cutoffs are pairwise disjoint.  Therefore,
\begin{equation}\label{eq:phi-overlap-v2}
 \#\mathcal F_s(\alpha)\leq1.
\end{equation}
Since
\[
 \supp p_{\theta,k}
 \subset
 \supp\bigl(\eta_\theta(\,\cdot-\theta)\bigr)
 \cup
 \supp\bigl(\phi(Q_s^2(\,\cdot-\theta))\bigr),
\]
we have
\[
 \left\{
 \theta\in\Lambda_s:p_{\theta,k}(\alpha)\neq0
 \right\}
 \subset
 \mathcal E_s(\alpha)\cup\mathcal F_s(\alpha).
\]
It follows from \eqref{eq:eta-overlap-v2},
\eqref{eq:phi-overlap-v2}, and the estimate for
$\|p_{\theta,k}\|_\infty$ that
\begin{align*}
 |d_{s,k}(\alpha)|
 \le
 \sum_{\substack{\theta\in\Lambda_s
                  p_{\theta,k}(\alpha)\neq0}}
 |p_{\theta,k}(\alpha)|\lesssim2^{C_0s}2^{-k+C_1 2^{\kappa s}}.
\end{align*}
Taking the supremum over $\alpha\in\Torus$, we conclude that
\begin{equation*}
 \|d_{s,k}\|_\infty
 \lesssim2^{C_0s}2^{-k+C_1 2^{\kappa s}}.
\end{equation*}
Therefore,
\begin{align*}
\left\|
 \bigl(\mathcal T_{d_{s,k}}(f)\bigr)_{k\geq 4^s}
 \right\|_{L_2(\Aalg;\ell_\infty)}
&\qquad\lesssim\left(
 \sum_{k\geq 4^s}
 \|\F^{-1}[d_{s,k}\widehat f]\|_2^2
 \right)^{1/2}\\
&\qquad\lesssim2^{C_0s+C_1 2^{\kappa s}}
 \left(\sum_{k\ge4^s}2^{-2k}\right)^{1/2}\|f\|_2\\
&\qquad\lesssim2^{C_0s+C_1 2^{\kappa s}-4^s}\|f\|_2.
\end{align*}
Since $\kappa<2$, we have $2^{\kappa s}=o(4^s)$.  It follows that
\[
 2^{C_0s+C_1 2^{\kappa s}-4^s}
 \lesssim2^{-C_{-1} s
}\]
for some $C_{-1}>0$. We obtain
\[
 \left\|
 \bigl(\mathcal T_{d_{s,k}}(f)\bigr)_{k\ge4^s}
 \right\|_{L_2(\Aalg;\ell_\infty)}
 \lesssim2^{-C_{-1}s}\|f\|_2.
\]

\medskip

\emph{Step 3: conclusion.}
By \eqref{eq:comparison-decomposition-v2}, the triangle inequality
gives
\begin{align*}
 &\left\|
 \left(
 \mathcal T_{L_{2^k}^{(\leq2^s)}-\widehat M_{s,k}}f
 \right)_{k\geq4^s}
 \right\|_{L_2(\Aalg;\ell_\infty)}
 \\
 &\qquad\lesssim \left(2^{-\delta_ds}+2^{-C_{-1}s}\right)\|f\|_2
\lesssim2^{-\gamma's}\|f\|_2,
\end{align*}
where
\[ \gamma'=\min\{\delta_0,C_{-1}\}>0.
\]
This proves \eqref{eq:comparison-v2}.
\end{proof}

We have now obtained both the uniform $L_p$-maximal estimate and the
decaying $L_2$-approximation estimates for the major-arc
multipliers.  These estimates will be used in the next section to
complete the proof of the main theorem.

\section{Proof of the main result--Theorem~\ref{thm1} }\label{sec:proof of main result}

With all the preparatory work in the previous sections, we are now able to conclude the main result by following Bourgain's strategy.

\begin{proof}[Proof of Theorem~\ref{thm1}]
We first prove the maximal inequality \eqref{mthm1}.
 It suffices to prove
\eqref{thm11}, namely,
\begin{equation}\label{thm111}
   \big\|(f*K_{2^k})_{k\in\mathbb N}\big\|_
   {L_p(\mathcal A_{\mathbb Z}(\mathcal N);\ell_\infty)}
   \lesssim_{p} \|f\|_{L_p(\mathcal A_{\mathbb Z}(\mathcal N))}.
\end{equation}

Since each convolution operator $f\mapsto f*K_{2^k}$ is positive,
the standard decomposition of an arbitrary element into four
positive elements reduces the proof to the case $f\geq0$.  We first
consider $1<p<2$.  The finitely many indices $k<4$ are bounded by
$C\|f\|_p$, since $\|f*K_{2^k}\|_p\leq\|f\|_p$.  Thus it remains to
consider $k\geq4$.

Given $k\in[4^s,4^{s+1})$, where $s$ is a positive integer,
\begin{equation*}
    f*K_{2^k}=\Big(f*K_{2^k}-f*M_{s,k}\Big)+\sum_{s^\prime=1}^sf*\Big(M_{s^\prime,k}-M_{s^\prime-1,k}\Big).
\end{equation*}
Here, $M_{0,k}:=0$.
For $k\ge4$, put $s(k)=\lfloor\log_4k\rfloor$. Then
$$\Big(\sum_{s^\prime=1}^{s(k)}f*\big(M_{s^\prime,k}-M_{s^\prime-1,k}\big)\Big)_{k\geq 4}=\sum_{s^\prime=1}^{\infty}\Big(1_{k\geq 4^{s'}}f*\big(M_{s^\prime,k}-M_{s^\prime-1,k}\big)\Big)_{k\geq 4}.$$
Consequently,
\begin{align}\label{Key-es2}
    \big\|(f*K_{2^k})_{k\in\mathbb N}\big\|_
    {L_p(\mathcal A_{\mathbb Z}(\mathcal N);\ell_\infty)}
    \lesssim(I)+(II),
    \end{align}
where
\begin{equation*}
    (I)=\sum_{s=1}^\infty
    \big\|\big(f*(K_{2^k}-M_{s,k})\big)_{4^s\leq k<4^{s+1}}\big\|_
    {L_p(\mathcal A_{\mathbb Z}(\mathcal N);\ell_\infty)}
\end{equation*}
and
\begin{equation*}
    (II)=\sum_{s^\prime=1}^\infty
    \big\|\big(f*(M_{s^\prime,k}-M_{s^\prime-1,k})\big)_{k\ge4^{s^\prime}}\big\|_
    {L_p(\mathcal A_{\mathbb Z}(\mathcal N);\ell_\infty)}.
\end{equation*}
Choose $p_0$ such that $1<p_0<p$.  By Lemma \ref{lem:M-p-v2},
\begin{equation}\label{Key-es3}
    \big\|(f*M_{s^\prime,k})_{k\ge4^{s^\prime}}\big\|_
    {L_{p_0}(\mathcal A_{\mathbb Z}(\mathcal N);\ell_\infty)}
    \lesssim\|f\|_{p_0}.
\end{equation}Hence, by the triangle inequality and \eqref{Key-es3},
\begin{equation}\label{Key-es5}
    \big\|\big(f*(M_{s^\prime,k}-M_{s^\prime-1,k})\big)_{k\ge4^{s^\prime}}\big\|_
    {L_{p_0}(\mathcal A_{\mathbb Z}(\mathcal N);\ell_\infty)}
    \lesssim\|f\|_{p_0}.
\end{equation}
The interval $[4^s,4^{s+1})$ is the union of
$[4^s,2\cdot4^s)$ and $[2\cdot4^s,4\cdot4^s)$.  Applying
Proposition \ref{Key1-es} with $N_0=2^{4^s}$ and
$N_0^2=2^{2\cdot4^s}$, respectively, and absorbing the endpoints into
the constant, we obtain the following estimate, since
$\log\log(2^{4^s})\simeq s$:
\begin{equation}\label{Key-es4}
    \big\|(f*K_{2^k})_{4^s\leq k<4^{s+1}}\big\|_
    {L_{p_0}(\mathcal A_{\mathbb Z}(\mathcal N);\ell_\infty)}
    \lesssim s\|f\|_{p_0}.
\end{equation}
Moreover, \eqref{Key-es3} and \eqref{Key-es4} give
\begin{equation}\label{Key-es6}
    \big\|\big(f*(K_{2^k}-M_{s,k})\big)_{4^s\leq k<4^{s+1}}\big\|_
    {L_{p_0}(\mathcal A_{\mathbb Z}(\mathcal N);\ell_\infty)}
    \lesssim s\|f\|_{p_0}.
\end{equation}
By Lemma \ref{lem:L2decay-v2}, for $s^\prime\geq2$ we have
\begin{equation}\label{Key-es7}
   \big\|\big(f*(M_{s^\prime,k}-M_{s^\prime-1,k})\big)_{k\ge4^{s^\prime}}\big\|_
   {L_2(\mathcal A_{\mathbb Z}(\mathcal N);\ell_\infty)}
   \lesssim2^{-\delta s^\prime}\|f\|_{2}.
\end{equation}
The same lemma also gives, for every $s\geq1$,
\begin{equation}\label{Key-es8}
     \big\|\big(f*(K_{2^k}-M_{s,k})\big)_{4^s\leq k<4^{s+1}}\big\|_
     {L_2(\mathcal A_{\mathbb Z}(\mathcal N);\ell_\infty)}
     \lesssim 2^{-\delta s}\|f\|_{2}.
\end{equation}
The case $s^\prime=1$ in \eqref{Key-es7} follows from Lemma
\ref{lem:M-p-v2} with $p=2$, after increasing the implicit constant.

Let $0<\theta<1$ be determined by
\[
 \frac1p=\frac{1-\theta}{p_0}+\frac\theta2.
\]
Interpolating \eqref{Key-es5} with \eqref{Key-es7}, we obtain
\begin{equation}\label{Key-es9}
   \big\|\big(f*(M_{s^\prime,k}-M_{s^\prime-1,k})\big)_{k\ge4^{s^\prime}}\big\|_
   {L_p(\mathcal A_{\mathbb Z}(\mathcal N);\ell_\infty)}
   \lesssim2^{-\delta_ps^\prime}\|f\|_{p}.
\end{equation}
Similarly, interpolation between \eqref{Key-es6} and
\eqref{Key-es8} first gives
\begin{equation}\label{Key-es10}
     \big\|\big(f*(K_{2^k}-M_{s,k})\big)_{4^s\leq k<4^{s+1}}\big\|_
     {L_p(\mathcal A_{\mathbb Z}(\mathcal N);\ell_\infty)}
     \lesssim s^{1-\theta}2^{-\theta\delta s}\|f\|_{p}
     \lesssim 2^{-\delta_ps}\|f\|_{p},
\end{equation}
where one may take any $0<\delta_p<\theta\delta$.  Combining
\eqref{Key-es2}, \eqref{Key-es9}, and \eqref{Key-es10}, and summing
the resulting geometric series, we deduce
\begin{equation*}
     \big\|(f*K_{2^k})_{k\in\mathbb N}\big\|_
     {L_p(\mathcal A_{\mathbb Z}(\mathcal N);\ell_\infty)}
     \lesssim \|f\|_p.
\end{equation*}

For $p=2$, the same conclusion follows directly from
\eqref{Key-es2}, \eqref{Key-es7}, and \eqref{Key-es8}.  Finally, for
$2<p<\infty$, interpolate the $L_2$ estimate with the evident bound
\[
 \big\|(f*K_{2^k})_{k\in\mathbb N}\big\|_
 {L_\infty(\mathcal A_{\mathbb Z}(\mathcal N);\ell_\infty)}
 \leq \|f\|_\infty.
\]
This proves \eqref{thm111} for every $1<p<\infty$, and hence
establishes \eqref{mthm1}.

Since the maximal inequality holds in $L_p$, to prove the b.a.u. convergence for $f\in L_p$, it suffices to do it only when $f\in L_2$.  To do so, by the  the discussion in
\cite[Section~4.1]{HongWang+arXiv2026}, it suffices to prove an analog of arguments of
\cite[Proposition~4.4]{HongWang+arXiv2026}. More precisely,  it suffices to show that for  fixed $0<\varepsilon<1/4$, there is a sequence $c_J\to 0$ such that  with
\[
 \mathbb D_\varepsilon
 =
 \{\lfloor(1+\varepsilon)^k\rfloor:k\geq1\},
 \qquad
 I_i=\mathbb D_\varepsilon\cap[N_i,N_{i+1}),
\]
where $N_i\in\mathbb D_\varepsilon$ and $N_{i+1}>2N_i$ and any $f\in L_2$:
\begin{equation}\label{eq:quantitative-convergence}
 \frac1{J+1}\sum_{i=0}^J
 \left\|
 \bigl(f*K_N-f*K_{N_i}\bigr)_{N\in I_i}
 \right\|_{L_2(\mathcal A_{\mathbb Z}(\mathcal N);\ell_\infty)}^2
 \leq c_J\|f\|_2^2.
\end{equation}

Fix $s\geq1$. Decompose, using \eqref{eq:LN-v2} and \eqref{eq:defLNs},
\[
 \widehat K_N
 =
 (\widehat K_N-L_N)
 +L_N^{(\leq2^s)}
 +(L_N-L_N^{(\leq2^s)}).
\]
Let $T_N^{(1)}$, $T_{s,N}^{(2)}$, and $T_{s,N}^{(3)}$
be the operators with these three multipliers, respectively.
Then
\[
 f*K_N
 =
 T_N^{(1)}f+T_{s,N}^{(2)}f+T_{s,N}^{(3)}f.
\]
For the first term, \eqref{eq:scalar-error-v2},
and Plancherel's theorem give
\begin{align*}
 \sum_{i=0}^J
 \left\|
 \bigl(T_N^{(1)}f-T_{N_i}^{(1)}f\bigr)_{N\in I_i}
 \right\|_{L_2(\mathcal A_{\mathbb Z}(\mathcal N);\ell_\infty)}^2
 \lesssim_{\varepsilon,P}
 \left(
 \sum_{N\in\mathbb D_\varepsilon}(\log N)^{-2\sigma}
 \right)\|f\|_2^2
 \lesssim_P\|f\|_2^2,
\end{align*}
since $\sigma>1/2$.
For the third term, summing the bounds in
Lemma~\ref{lem:layer-v2} over dyadic $u>2^s$ yields
\[
 \sum_{i=0}^J
 \left\|
 \bigl(T_{s,N}^{(3)}f-T_{s,N_i}^{(3)}f\bigr)_{N\in I_i}
 \right\|_{L_2(\mathcal A_{\mathbb Z}(\mathcal N);\ell_\infty)}^2
 \lesssim_P
 (J+1)2^{-2\gamma s}\|f\|_2^2.
\]
It remains to estimate $T_{s,N}^{(2)}$. Its multiplier is a
finite sum of terms
\[
 S(\theta)w_N(\alpha-\theta)\chi(\alpha-\theta),
\]
where $\theta$ and $\chi$ are independent of $N$.
We compare $w_N$ with the corresponding expression for the
leading monomial $a_dt^d$:
\[
 w_N^{(0)}(\beta)
 :=
 \frac1N\int_0^N
 \exp(-2\pi i\beta a_dt^d)\,dt.
\]
 Estimate \eqref{eq:wdec-v2} gives
\[
 |w_N(\beta)-w_N^{(0)}(\beta)|
 \lesssim_P
 \min\bigl\{
 |\beta|N^{d-1},(N^d|\beta|)^{-1/d}
 \bigr\},
 \qquad \beta\neq0.
\]
Consequently,
\[
 \|w_N-w_N^{(0)}\|_\infty
 \lesssim_P N^{-1/(d+1)}.
\]
For each fixed $\theta$ and $\chi$,
\cite[Corollary~4.8 and Remark~4.9]{HongWang+arXiv2026}
gives the quantitative maximal estimate over the intervals
$I_i$ for the multiplier
\[
 S(\theta)w_N^{(0)}(\alpha-\theta)\chi(\alpha-\theta),
\]
with a constant independent of $J$ and $(N_i)$. For the difference $w_N-w_N^{(0)}$, the same
calculation as for $T_N^{(1)}$, we have
\[
 \left(
 \sum_{N\in\mathbb D_\varepsilon}
 \|w_N-w_N^{(0)}\|_\infty^2
 \right)\|f\|_2^2
 \lesssim_{P}
 \left(
 \sum_{N\in\mathbb D_\varepsilon}N^{-2/(d+1)}
 \right)\|f\|_2^2
 \lesssim_{\varepsilon,P}\|f\|_2^2.
\]
Combining this bound with the estimate for $w_N^{(0)}$
and summing over the finitely many terms, we obtain
\[
 \sum_{i=0}^J
 \left\|
 \bigl(T_{s,N}^{(2)}f-T_{s,N_i}^{(2)}f\bigr)_{N\in I_i}
 \right\|_{L_2(\mathcal A_{\mathbb Z}(\mathcal N);\ell_\infty)}^2
 \lesssim_{\varepsilon,P,s} \|f\|_2^2,
\]
where the constant is independent of $J$ and $(N_i)$.

Combining these three bounds, we obtain
\[
 \frac1{J+1}\sum_{i=0}^J
 \left\|
 \bigl(f*K_N-f*K_{N_i}\bigr)_{N\in I_i}
 \right\|_{L_2(\mathcal A_{\mathbb Z}(\mathcal N);\ell_\infty)}^2
 \leq
 \left(
 \frac{C_{\varepsilon,P,s}}{J+1}
 +C_P2^{-2\gamma s}
 \right)\|f\|_2^2.
\]
Letting $J\to\infty$ and then $s\to\infty$ proves
\eqref{eq:quantitative-convergence}.
The desired b.a.u. convergence now follows from the discussion in
\cite[Section~4.1]{HongWang+arXiv2026}.
This completes the proof.
\end{proof}

\appendix

\section{A noncommutative Burkholder's theorem for $p=1$}

Burkholder in \cite{Bur64} obtained a probabilistic analogue of Stein's result.
For this one, it is possible to extend it to noncommutative setting for $p=1$.
What is challenging again here is to find a suitable tool to replace the Borel-Cantelli Lemma. We follow closely \cite[Theorem 1]{Bur64} .

Two noncommutative probability spaces $(\mathcal M_1,\tau_1)$ and
$(\mathcal M_2,\tau_2)$ sitting in a bigger one $(\mathcal M,\tau)$
(in particular $\tau_{|\mathcal M_i}=\tau_i$, so we skip the indices
for the traces) are said to be independent if for all
$x_1\in L_1(\mathcal M_1)$ and $x_2\in L_2(\mathcal M_2)$, we have
$\tau(x_1x_2)=\tau(x_1)\tau(x_2)$. The obvious way to make two spaces
independent is to take their tensor product or reduced free product. A
sequence of noncommutative probability spaces $(\mathcal M_n)_n$ of
$(\mathcal M,\tau)$ is independent if, for all $n$, $\mathcal M_{n+1}$
is independent with the algebra generated by $\mathcal M_i$'s for $i\leq n$.

Given a sequence $F=(f_n)$ in a noncommutative probability space (more precisely in $L_1$), we
write $(\mathcal M_F, \tau_F)$ for the algebra it generates and
$\tau_F$ is just the restriction of the ambiant trace. For two
sequences $F=(f_n)$ and $G=(f_n)$ in possibly different noncommutative probability spaces, we say they have the same distribution
and write $F\sim G$ if there is a trace preserving isomorphism $\pi$ from
$(\mathcal M_F, \tau_F)$ into $(\mathcal M_G, \tau_G)$ with $\pi(f_n)=g_n$ (with the natural extension to $L_1$).

We denote by $\mathcal D$ the collection of all positive integrable sequences
$(f_n)_n$ over noncommutative probability spaces (that is
$L_1(\mathcal M,\tau)_+$, where $(\mathcal M,\tau)$ is finite). A
collection $\mathcal C\subset D$ is stochastically convex if $F_k=(f_{k,n})_n$ is a sequence in $\mathcal C$, then there is another sequence $G_k\in \mathcal C$ such that

\begin{enumerate}
  \item[(i)] the algebra generated by $G_k$'s are independent;
  \item[(ii)] for all $k\geq 1$ $F_k\sim G_k$;
  \item[(iii)] if $(a_k)_{k\geq 1}$ is a sequence in $\mathbb R^+$ with $\sum_k a_k=1$, then there is a sequence  $H\in \mathcal C$ such that $h\sim \sum_{k=1}^\infty a_k G_k$.
\end{enumerate}

As in \cite{Bur64}, we say that $\mathcal C$ satisfies the finiteness
condition if for all $F=(f_n) \in \mathcal C$, there is a projection $p$ (in the ambiant algebra) with $\tau(p)>0$ and $\sup_n \|pf_np\|_\infty<\infty$.

\begin{theorem}
  If $\mathcal C$ is stochastically convex with the finiteness condition, then there is some $K>0$ such that for all $F\in \mathcal C$, $\|F\|_{\Lambda_{1}(\mathcal M_F;\ell_\infty)}\leq K$.
\end{theorem}

\begin{proof}
  As in  \cite{Bur64}, if the conclusion does not hold then we may find a sequence $\lambda_k\to \infty$, $a_k>0$ with $\sum a_k=1$ and $G_k=(g_{k,n})_n\in \mathcal C$ such that $\sum_{k=1}^\infty \Phi_{\rm inf}(\sup^+_na_kg_{k,n}>\lambda_k)=\infty$ and the $(G_k)$ are tensor independent.

  Let $\mathcal M_k= \mathcal M_{G_k}$ for $k\geq 1$, we can see them in a bigger $(\mathcal M,\tau)$ as they are independent. There are conditional expectations $\mathbb E_k : \mathcal M \to \mathcal M_k$. We will also use that
  independence forces the spaces
$L_2(\mathcal M_k)\cap 1^\bot$ to be orthogonal in $L_2(\mathcal M)$.

  Let $X=(x_n)_n=\sum_{k=1}^\infty
  a_kG_k$, that is $x_n=\sum_{k=1}^\infty a_kg_{k,n}$.
  To get a contraction with the finiteness condition, we
  want to show that for any projection $p\in \mathcal
  M$ with $\tau(p)=t>0$ we have $\sup_n \|px_np\|=\infty$.

 Set
  $f_k=\mathbb E_k(p)$, then $f_k-t$ has mean 0 and by independence
  $\sum_{k=1}^\infty \|f_k-t\|_2^2\leq \| p-t\|_2^2<\infty$.

  Put $p_k=1_{f_k\geq t/2}$ and $p_k^\bot=1-p_k=1_{t-f_k> t/2}$. We have $t^2/4\tau(p_k^\bot)\leq \|f_k-t\|^2$, thus $\sum_k \tau(p_k^\bot)<\infty$. This obviously implies that $I=\{k\geq 1 \;|\; \tau(p_k^\bot)<\Phi_{\rm inf}(\sup^+_na_kg_{k,n}>\lambda_k)\}$ is infinite.

  Now, since $x_n\geq a_kg_{k,n}$, we have
  \begin{equation*}\|x_n^{1/2}p\|_\infty\geq
  \|(a_kg_{k,n})^{1/2}p\|_\infty\geq \| \mathbb E_k((a_kg_{k,n})^{1/2}p)\|_\infty=\|(a_kg_{k,n})^{1/2}f_k\|_\infty.\end{equation*}
  But $f_k\geq t/2 p_k$ and they commute so $t/2p_k=f_kc$ for some contraction $c$
  and we conclude, for all $k,n\geq 1$
  \begin{equation}\label{normxn}\|px_np\|_\infty\geq t^2/4\|p_k a_k g_{k,n}p_k\|_\infty.\end{equation}

  Fix $k\geq 1$, if for all $n\geq 1$, we have
  $\|p_k a_k g_{k,n}p_k\|_\infty\leq \lambda_k$, then by definition of $\Phi_{\inf}$ in \eqref{phi_inf},  $\Phi_{\rm inf}(\sup^+_na_kg_{k,n}>\lambda_k)\leq \tau(p_k^\bot)$ and thus $k\notin I$. It follows that $\sup_n \|p_ka_k g_{k,n}p_k\|\geq \lambda_k$ for $k\in I$. Since $I$ is infinite and $\lambda_k\to \infty$, \eqref{normxn} yields $\sup_n \|px_np\|=\infty$ and the proof is over.
\end{proof}

\noindent\textbf{Acknowledgements.}
Part of this work was carried out during the first author’s visit to the University of Caen  in 2025, supported by a poste rouge from the CNRS. The first author would like to thank \'Eric Ricard and the Department of Mathematics for their warm hospitality. 
This work was partially supported by the National Natural Science
Foundation of China (grant nos.~12325105,
W2441002, and W2611005).

\noindent\textbf{AI statement.}
In addition to assisting with English-language polishing, the \LaTeX{} formatting, and preliminary checks of mathematical consistency, GPT-5.6 Sol suggested an approach to the treatment of \eqref{ai} for integer $r$ in terms of triangular truncation. This inspired the authors, but the precise formulation and the proof of the multilinear Doob's maximal inequality are given independently by the authors and do not rely on AI-generated arguments. Apart from this suggestion, all mathematical ideas, results, arguments, and proofs in the manuscript were developed independently by the authors without reliance on AI assistance. 

\bibliographystyle{amsplain}

\end{document}